\documentclass[a4paper,12pt,reqno,oneside]{amsart}
\usepackage[T1]{fontenc}
\usepackage[utf8]{inputenc}
\usepackage[left=1.5cm,right=1.5cm]{geometry}
\usepackage[activate={true,nocompatibility},final,tracking=true,kerning=true,spacing=true,factor=1100,stretch=10,shrink=2]{microtype}
\usepackage[backend=biber,style=numeric,backref=true,giveninits=true,sortcites=true]{biblatex}
\usepackage{amssymb,amsmath}
\usepackage{hyperref}
\usepackage[capitalise,nameinlink]{cleveref}
\usepackage[percent]{overpic}
\usepackage{mathtools}
\usepackage{indentfirst}
\usepackage{enumitem}
\usepackage{accents}
\usepackage{subcaption}

\usepackage{tikz,tikz-cd,tikz-3dplot}
\usepackage{tikz-cd}
\usetikzlibrary{calc}

\hypersetup{colorlinks,linkcolor=purple,citecolor=magenta,urlcolor=magenta}

\renewbibmacro{in:}{}
\DefineBibliographyStrings{english}{
  backrefpage={cited on p.},
  backrefpages={cited on pp.}
}

\crefformat{figure}{#2Figure~#1#3}
\newcommand{\crefitem}[2]{\hyperref[#2]{\cref*{#1}(\ref*{#2})}}

\theoremstyle{plain}
\newtheorem{theorem}{Theorem}
\newtheorem*{theorem*}{Theorem}
\newtheorem*{conjecture}{Conjecture}
\newtheorem{lemma}{Lemma}
\newtheorem*{lemma*}{Lemma}

\newtheorem*{question}{Question}

\makeatletter
\def\re@title{}
\newtheorem*{re@theorem}{\re@title}
\newenvironment{retheorem}[1]{%
 \def\re@title{\cref{#1}}%
 \begin{re@theorem}}%
 {\end{re@theorem}}
\makeatother

\theoremstyle{remark}

\theoremstyle{definition}
\newtheorem*{definition}{Definition}

\setlist{listparindent=\parindent,parsep=\parskip}

\newcommand{\from}{\colon}
\newcommand{\rest}[1]{\nobreak\raisebox{-.1ex}{$|$}_{#1}}

\newcommand{\defeq}{\coloneq}

\newcommand{\inter}[1]{\accentset{\smash{\raisebox{-0.12ex}{$\scriptstyle\circ$}}}{#1}\rule{0pt}{2.3ex}}
\newcommand{\border}[1]{\accentset{\smash{\raisebox{-0.12ex}{$\scriptstyle\bullet$}}}{#1}\rule{0pt}{2.3ex}}

\newcommand{\collapse}[2]{\mathrel{\substack{\\[-0.5em]#1}{\searrow_{#2}}}}
\newcommand{\collapsible}{\searrow}
\newcommand{\tendsto}{\longrightarrow}

\newcommand{\A}{\mathbb{A}}
\newcommand{\N}{\mathbb{N}}

\DeclareMathOperator{\pr}{pr}
\DeclareMathOperator{\id}{id}
\DeclareMathOperator{\tinter}{int}
\DeclareMathOperator{\sk}{sk}

\author{Alexey Gorelov}
\title{Geometry of collapsing and free deformation retraction}

\begin{document}

\begin{abstract}
  We show that a compact polyhedron $P$ collapses to a subpolyhedron $Q$ if and only if it admits a
  piecewise-linear free deformation retraction onto $Q$. We also consider further possibilities for
  invariant characterisations of collapsibility in terms of metrics; in this connection, we provide
  a partial correction to Isbell's claim that every injective metric space is freely contractible,
  and present a counterexample to a step in the original argument.
\end{abstract}

\maketitle

\section{Introduction}

One of the central concepts in piecewise-linear topology is \emph{collapsibility}. Its standard
definition is formulated for simplicial complexes and, in the case of polyhedra, is expressed by
requiring the existence of a triangulation that is collapsible. Although several other equivalent
definitions of collapsibility exist (one of which we discuss below), all of them depend on placing
additional combinatorial structure on the polyhedron. This reliance on the existence of specific
combinatorial structures complicates its use in the study of other topological invariants.

For example, one of the main open problems in piecewise linear topology links collapsibility with
contractibility:
\begin{conjecture}[{Zeeman conjecture~\cite{zeeman1963dunce}}]
  If $P$ is a contractible two-dimensional compact polyhedron then $P \times I$ is collapsible (to a
  point).
\end{conjecture}

The significance of this conjecture is highlighted by its connections to other central problems in
topology. For instance, its special case is equivalent to the three-dimensional Poincaré
conjecture~\cite{zeeman1963dunce,equiv_poincare}, proved by Perelman~\cite{perelman1,perelman2},
while another special case is equivalent to the stable Andrews-Curtis conjecture, closely related to
the smooth four-dimensional Poincaré conjecture~\cite{matveev_ac}. As far as the author knows, the
latter two problems remain open. Interested readers may refer to subsection 1.3.4 in~\cite{matveev}
for proofs of both of these equivalences.

In~\cite{isbell64}, Isbell introduced the notion of a \emph{free deformation retraction}: a strict
deformation retraction $h \from X \times I \to X$ is called free if
$h(h(x, t), s) = h(x, \max(t, s))$ for all $s, t \in I$. He proved that a two-dimensional compact
polyhedron $P$ is collapsible if and only if it admits a free deformation retraction to a point.
Independently, Piergallini~\cite{piergallini85} reintroduced this concept under the name
\emph{topological collapsing}. It follows from~\cite{tymchatyn,piergallini85} that the same
equivalence holds for three-dimensional manifolds with boundary.\footnote{{In view of~\cite[Theorem
    23, Corollary 1]{whitehead} or~\cite[Theorem 8, Corollary 1]{zeeman_seminar} and Moise's
    theorem~\cite{moise}, this statement is equivalent to saying that freely contractible
    3-manifolds with boundary are 3-balls.}} However, in dimensions five and higher there exist
non-collapsible polyhedra that are topological balls and hence freely
contractible~\cite{berstein1978contractible}. Thus, in general, collapsibility and free
contractibility are not equivalent.

We show that the equivalence holds in general if the deformation retraction is required to be
piecewise linear. Our main result is the following:
\begin{retheorem}{theorem:main_theorem}
  Let $P$ be a compact polyhedron and let $Q \subseteq P$ be a subpolyhedron. Then $P$ collapses to
  $Q$ if and only if there exists a piecewise-linear free deformation retraction of $P$ onto $Q$.
\end{retheorem}

One can look at this result as an attempt to characterise collapsibility in an invariant
way.\footnote{Note that the piecewise linearity of a map can be defined in an invariant way,
  see~\cite{rourke}.}

It should be noted that this theorem was already stated in~\cite{piergallini85}, but the proof given
there contains a substantial gap: it assumes the existence of triangulations $(T, T_P)$ of
$P \times I$ and $P$ with respect to which both the projection $\pr_P$ and the piecewise-linear free
deformation retraction are simplicial. In general, such triangulations need not exist: for two
piecewise-linear maps $f \from P \to Q$ and $g \from P \to M$ there may be no triangulations of $P$,
$Q$, and $M$ with respect to which both $f$ and $g$ are simplicial, see Example~1 to Theorem~1
in~\cite{zeeman_seminar} or the examples after Lemma~2.17 in~\cite{rourke}.

Another topic considered in this work is the possibility of characterising collapsibility in metric
terms. In~\cite{isbell64}, Isbell claimed that every injectively metrizable metric space is freely
contractible to any of its points. As we explain in the final section, the proof given there
contains a substantial gap. We provide a corrected argument that works under the additional
assumption that the space is proper; in particular, it applies to compact polyhedra:
\begin{retheorem}{theorem:injective_metric}
  Every proper injective metric space is freely deformation contractible to each of its points.
\end{retheorem}

In view of the previous theorem, this naturally leads to the following question:
\begin{question}
  Given a compact polyhedron $P$, what additional conditions on an injective metric on $P$ ensure
  the existence of a piecewise-linear free contraction of $P$?
\end{question}

For cubic complexes, there are metric characterizations of (cubic) collapsibility. For instance, the
following theorem was stated by Melikhov in a talk at the Geometric Topology Seminar of the Steklov
Mathematical Institute:
\begin{theorem*}[\cite{melikhov_talk}]
  Let $C$ be a finite cubical complex, and let $d_{p}$ be the metric on $C$ induced by
  $\ell_{p}$-norms on the cubes of $C$ (considered as unit cubes). Then the following statements are
  equivalent:
  \begin{enumerate}[nosep]
  \item\label{st:median} $(C, d_{1})$ is a median metric space.
  \item\label{st:cat0} $(C, d_{2})$ is a CAT(0) space.
  \item\label{st:injective} $(C, d_{\infty})$ is an injective metric space.
  \item\label{st:collapses} $C$ is (cubically) collapsible.
  \end{enumerate}
\end{theorem*}

\begin{proof}
  The proof is distributed across several papers. Below we provide a diagram in which each
  implication is indicated together with a reference:

  \begin{tikzcd}[column sep=3.2cm,row sep=3.2cm,math mode=false]
    $(C, d_{1})$ is median \arrow[leftrightarrow]{r}{\cite{chepoi} and \cite[\S 3]{vandevel_book}}
    \arrow[leftrightarrow]{d}[swap]{\cite{vandevel_median_collapse}} & \arrow{dl}[swap]{\cite{adiprasito_benedetti,miesch}} $(C, d_{2})$ is CAT(0) \\
    $C$ is collapsible \arrow{r}{\cite{mai_tang}} & $(C, d_{\infty})$ is injective
    \arrow[leftrightarrow]{u}{\cite{miesch}}
  \end{tikzcd}$ $\\\qedhere
\end{proof}

The notion of cubical collapsibility used in this theorem differs from the simplicial collapsibility
considered in our work. Nevertheless, one can show that every simplicially collapsible polyhedron
admits a cubically collapsible cubulation, and conversely, the geometric realization of a cubically
collapsible cubical complex is simplicially collapsible (see~\cite[Section
4.1]{vandevel_median_collapse}). In this sense, the theorem above provides an answer to the question
posed earlier. However, it remains open whether a metric characterization in more invariant terms is
possible.

\subsection{Notation}

Throughout this text, we use capital letters $P$, $Q$, $M$, and $N$ to denote compact polyhedra, and
$A$, $B$, and $C$ to denote simplices. Lowercase letters are used for points and vertices. When no
confusion arises, we also use $B$ for balls; for example, $B^{n}$ is the $n$-ball, and $B(x, r)$ is
the metric ball of radius $r$ centered at $x$.

We write $K$, $L$, and $T$ for simplicial complexes. For a simplicial complex $K$ we denote its
underlying polyhedron by $|K|$. For a polyhedron $P$, we write $T(P)$ for a triangulation of $P$,
that is, a simplicial complex with $|T(P)| = P$. The notation $\sk_i K$ denotes the $i$-skeleton of
a simplicial complex $K$, namely the union of all simplices of dimension at most $i$.

For a simplex $A$, we denote by $\border{A}$ its relative boundary, that is, the union of all proper
faces of $A$. We write $\inter{A}$ for its relative interior $A \setminus \border{A}$. Given a
subspace $Y$ of a space $X$, we denote by $\partial Y$ and $\tinter{Y}$ the topological boundary and
topological interior of $Y$ in $X$, respectively.

\section{Preliminaries}

\subsection{Simplical complexes, polyhedra, and collapsibility}

In this subsection we recall the main notions and fix notation we use throughout the paper. We rely
primarily on~\cite{zeeman_seminar}; the reader may also refer to~\cite{rourke}.

A \emph{(finite affine) simplicial complex} is a finite nonempty collection $K$ of simplices in some
Affine space $\A^n$ such that:
\begin{enumerate}[nosep]
\item If $A \in K$, then every face of $A$ also belongs to $K$;
\item If $A, B \in K$, then either $A \cap B = \varnothing$, or $A \cap B$ is a common face of both $A$ and $B$.
\end{enumerate}

A \emph{(compact affine) polyhedron} $P$ is the underlying space of a simplicial complex $K$, namely
$P = |K| \subset \A^n$. The simplicial complex $K$ is then called a \emph{triangulation} of $P$. We
write $T(P)$ to denote a triangulation of $P$.

A continuous map $f \from P \to Q$ between polyhedra $P \subset \A^{n}$ and $Q \subset \A^{k}$ is
called \emph{piecewise-linear} if its graph in $\A^{n+k}$ is a polyhedron. Equivalently, $f$ is
piecewise-linear if there exist triangulations $T(P)$ and $T(Q)$ of $P$ and $Q$ such that $f$ is
\emph{simplicial}; this means that for every simplex $A \in T(P)$, its image $f(A)$ is a simplex in
$T(Q)$ and $f\rest{A}$ is affine.

Let $K$ be a simplicial complex, and let $A \in K$ be a simplex with a \emph{free face} $B \in K$,
that is, $B$ is a face of $A$ and is not contained in any simplex of $K$ other than $A$. An
\emph{elementary simplicial collapse} of $K$ across $A$ from $B$ is the removal of the simplices $A$
and $B$, resulting in the subcomplex $K' = K \setminus \{A, B\}$. We denote this operation by
$K \collapse{B}{A} K'$.

A simplicial complex $K$ \emph{collapses} to a subcomplex $S \subseteq K$, written
$K \collapsible S$, if there exists a finite sequence of elementary collapses
\[
  K = K_1 \collapse{B_1}{A_1} K_2 \collapse{B_2}{A_2} \cdots \collapse{B_{n-1}}{A_{n-1}} K_n = S.
\]
A polyhedron $P$ \emph{simplicially collapses} to a subpolyhedron $Q \subseteq P$, written
$P \collapsible Q$, if there exists a triangulation $T(P)$ of $P$ such that $Q$ is triangulated by a
subcomplex $T(Q) \subseteq T(P)$ and $T(P) \collapsible T(Q)$.

There is also a notion of \emph{polyhedral collapsibility} for polyhedra, which we will need later.
We say that there exists an \emph{elementary polyhedral collapse} from a polyhedron $P$ to a
subpolyhedron $Q$ if there exists an embedded piecewise-linear $n$-ball $B^{n} \subset P$ such that
\begin{enumerate}[nosep]
\item $P = Q \cup B^{n}$,
\item $Q \cap B^{n}$ is a piecewise-linear $(n-1)$-ball $B^{n-1}$ embedded into the boundary of
  $B^n$ (as a piecewise-linear manifold).
\end{enumerate}
We say that $P$ \emph{collapses} onto $Q$ \emph{polyhedrally} if there is a finite sequence of such
polyhedral collapses starting from $P$ and ending with $Q$.

It is clear that if $P$ collapses simplicially to $Q$ then it also collapses polyhedrally. The
converse also holds: if $P$ collapses polyhedrally to $Q$, then $P$ collapses simplicially to $Q$;
see, for example, Corollary~2 to Theorem~4 in Chapter~3 of~\cite{zeeman_seminar} for a proof. With this equivalence
in mind, we will generally not distinguish between simplicial and polyhedral collapsibility for
polyhedra and will use the same notation $P \collapsible Q$ for both notions.

\subsection{Free deformation retraction}

In this subsection, we define a free deformation retraction and discuss its main properties.

\begin{definition}
  A continuous map $h \from X \times I \to X$ is called a \emph{free deformation retraction} of a
  topological space $X$ onto a subspace $Y \subseteq X$ if the following conditions hold:
  \begin{enumerate}[nosep]
  \item\label{PLretr_def:retr} $h$ is a strong deformation retraction of $X$ onto $Y$, that is
    \[
      h\rest{X \times \{0\}} = \id_{X}, \qquad h(X \times \{1\}) = Y, \qquad h\rest{Y \times \{t\}} = \id_Y,\ \forall t \in I.
    \]
  \item\label{PLretr_def:free} $h(h(x, s), t) = h(x, \max(s, t))$ for all $x \in X$, $s, t \in I$.
  \end{enumerate}
  Setting $h_t(x) \defeq h(x, t)$, the freeness condition~(\ref{PLretr_def:free}) can be expressed
  as $h_t \circ h_s = h_{\max(t, s)}$.
\end{definition}

To build some intuition, we now establish several basic properties of free deformation retractions.

\begin{lemma*}
  Let $h \from X \times I \to X$ be a free deformation retraction of $X$ onto $Y$, and let
  $\gamma_{x}(t) \defeq h(x, t)$ denote the trajectory of a point $x$ during the retraction. Then:
  \begin{enumerate}[nosep]
  \item For every $t \geq s$, $h_{s}\rest{h_{t}(X)} = \id_{h_{t}(X)}$ and $h_{t}(X) \subseteq h_{s}(X)$.
  \item For every $x \in X$ and every $p \in \gamma_{x}(I)$, let
    $t_{p} \defeq \min \gamma_{x}^{-1}(p)$. Then $\gamma_{p}(I) = \gamma_{x}([t_{p}, 1])$. Moreover,
    $\gamma_{p}([0, t_{p}]) = p$ and $\gamma_{p}\rest{[t_{p}, 1]} = \gamma_{x}\rest{[t_{p}, 1]}$.
  \item Define a partial order $\prec$ on $X$ by setting $x \prec y$ whenever $\gamma_{x}(I) \subset
    \gamma_{y}(I)$. Then:
    \begin{enumerate}[nosep]
    \item $(X, \prec)$ is a pseudotree: for every $x$, the lower closure
      $\downarrow\!\{ x \} \defeq \{ x' \mid x' \prec x \}$ is totally ordered.
    \item $Y \subseteq X$ is precisely the set of the minimal elements with respect to $\prec$.
    \end{enumerate}
  \end{enumerate}
\end{lemma*}

\begin{proof}
  For every $t \geq s$ we have $h_s \circ h_t = h_t$, which implies
  $h_{s}\rest{h_{t}(X)} = \id_{h_{t}(X)}$ and hence $h_{t}(X) \subseteq h_{s}(X)$.

  Next, fix $x \in X$ and $p \in \gamma_{x}(I)$, and let $t_{p} = \min \gamma_{x}^{-1}(p)$. Then for
  all $s \in I$ we have
  \[
    \gamma_{p}(s) = h(h(x, t_{p}), s) = \gamma_{x}(\max(s, t_{p})) =
    \begin{cases}
      \gamma_{x}(t_{p}) = p, &s \leq t_{p}, \\
      \gamma_{x}(s) &s \geq t_{p}. \\
    \end{cases}
  \]
  This gives $\gamma_{p}([0, t_{p}]) = p$ and
  $\gamma_{p}\rest{[t_{p}, 1]} = \gamma_{x}\rest{[t_{p}, 1]}$, and hence
  $\gamma_{p}(I) = \gamma_{x}([t_{p}, 1])$.

  Now, consider the partial order $\prec$ on $X$. Take $p$ and $q$ from the lower closure
  $\downarrow\!\{ x \}$ of some $x \in X$. Then $\gamma_{p}(I) = \gamma_{x}([t_{p}, 1])$ and
  $\gamma_{q}(I) = \gamma_{x}([t_{q}, 1])$. If $p$ and $q$ are distinct, one of these segments
  strictly contains the other, making $p$ and $q$ comparable. Thus, $\downarrow\!\{x\}$ is totally
  ordered.

  Finally, if $y \in Y$, then $\gamma_{y}(I) = \{y\}$, so $y$ is minimal under $\prec$. Conversely,
  for any $x \in X \setminus Y$, we have $x \succ \gamma_{x}(1) \in Y$, hence $Y$ is exactly the set
  of minimal elements.
\end{proof}

Thus, a free deformation retraction defines a sort of tree-like structure on $X$ with roots in $Y$,
and contracting the tree starting from the leaves yields the original retraction.

Two ``consecutive'' free deformation retractions can be naturally combined into a single one.

\begin{lemma}\label{lemma:join_two_retractions}
  Let $Z \subseteq Y \subseteq X$ be topological spaces, and let $f \from X \times I \to X$ and
  $g \from Y \times I \to Y$ be free deformation retractions of $X$ onto $Y$ and of $Y$ onto $Z$,
  respectively. Then there exists a free deformation retraction $h \from X \times I \to X$ of $X$
  onto $Z$. Moreover, if $X$, $Y$, and $Z$ are polyhedra and $f$ and $g$ are piecewise-linear, then
  $h$ can be chosen piecewise-linear as well.
\end{lemma}

\begin{proof}
  Define $h \from X \times I \to X$ by
  \[
    h(x, t) = \begin{cases} f(x, 2t), &t\leq \frac12, \\ g(f(x, 1), 2t-1), &t\geq \frac12. \end{cases}
  \]

  Clearly, $h$ is a strong deformation retraction of $X$ onto $Z$. If $f$ and $g$ are
  piecewise-linear, then $h$ is piecewise-linear as well. It remains to verify the freeness
  condition. If $\max(s, t) \leq \frac{1}{2}$ or $\min(s, t) \geq \frac{1}{2}$, this follows
  directly from the freeness of $f$ and $g$, respectively. For the case $s \leq \frac{1}{2} \leq t$, we have
  \[
  \begin{aligned}
    h_s(h_t(x)) &= h(h(x, t), s) = f(g(f(x, 1), 2t-1), 2s) = g(f(x, 1), 2t-1) = h_t(x), \\
    h_t(h_s(x)) &= h(h(x, s), t) = g(f(f(x, 2s), 1), 2t-1) = g(f(x, 1), 2t-1) = h_t(x),
  \end{aligned}
  \]
  showing that $h$ is free.
\end{proof}

\subsection{Shadows in cylinders}

Throughout the text, we will work extensively with subsets of cylinders. This subsection introduces
a few basic concepts that will be used later, while the next section discusses cylindrical
triangulations and cylinderwise collapsing.

\begin{definition}\label{definition:shadow}
  Let $Y \subseteq X \times I$. The \emph{lower shadow} of $Y$ is
  \[
    S_{-}(Y) \defeq \{ (x, t) \in X \times I \mid \exists (x, s) \in Y,\ t \leq s \}.
  \]
  The \emph{upper shadow} of $Y$ is
  \[
    S_{+}(Y) \defeq \{ (x, t) \in X \times I \mid \exists (x, s) \in Y,\ t \geq s \},
  \]
  and the \emph{total shadow} of $Y$ is
  \[
    S(Y) \defeq \{ (x, t) \in X \times I \mid \exists (x, s) \in Y \}.
  \]

  A set $Y \subset X \times I$ is called \emph{downward closed} (respectively \emph{upward closed}) if
  $S_{-}(Y) = Y$ (respectively $S_{+}(Y) = Y$).
\end{definition}

There is a natural partial order on $X \times I$ defined as $(x, s) \leq (y, t)$ if and only if
$x = y$ and $s \leq t$. With respect to this order, the lower and upper shadows of a set $Y$ are
simply its lower and upper closures, respectively.

The following statements immediately follow from the definition:
\begin{lemma}\label{lemma:shadows_properties}
  Let $Y \subseteq X \times I$. Then:
  \begin{enumerate}[nosep]
  \item $Y$ is downward closed if and only if for every point $(y, t) \in Y$, the segment $\{y\} \times [0, t]$ is contained in $Y$.
  \item $Y$ is upward closed if and only if for every point $(y, t) \in Y$, the segment
    $\{y\} \times [t, 1]$ is contained in $Y$.
  \item The complement of a downward closed set in $X \times I$ is upward closed, and vice versa.
  \end{enumerate}
\end{lemma}

\section{Cylinderwise collapsing}

In this section, we discuss cylinderwise collapsing, as introduced in Chapter 7
of~\cite{zeeman_seminar}. The main results of this section follow those in the reference, but we
include the proofs for completeness and because several concepts and intermediate results will be
used later in the text.

We begin by establishing a few preliminary statements from affine geometry.

\begin{lemma}\label{l:lines_simplex}
  Let $A$ be an $n$-simplex embedded in $\A^{n}$, and let $L$ be a line in $\A^{n}$. Suppose that
  $L \cap A$ contains a point $p$ not lying in $\sk_{n-2} A$. Then $L \cap A$ is a non-degenerate
  segment.
\end{lemma}

\begin{proof}
  Since both $L$ and $A$ are convex, their intersection $L \cap A$ is convex and hence either empty,
  a point, or a segment. As $p \notin \sk_{n-2} A$, the point $p$ must belong to the interior of
  $A$, or to the interior of a codimension-one face $B$ of $A$.

  If $p \in \inter{A}$, then $p$ lies in the relative interior of $L \cap A$, so the intersection
  cannot be a single point. Hence $L \cap A$ is a non-degenerate segment.

  Suppose now that $p \in \inter{B}$ for some codimension-one face $B$. Choose barycentric
  coordinates $(\lambda_0, \dots, \lambda_n)$ on $\A^{n}$ with respect to the vertices of $A$ so
  that $B$ is given by $\lambda_0 = 0$. Then $p = (0, \lambda_1, \dots, \lambda_n)$ with
  $\lambda_i > 0$ for all $i = 1, \dots, n$ and $\sum_{i=1}^n \lambda_i = 1$. Let
  $(\alpha_0, \dots, \alpha_n)$ be the direction vector of $L$ in these coordinates, satisfying
  $\sum_{i=0}^n \alpha_i = 0$. After rescaling, we may assume $\alpha_{0} \geq 0$ and
  $\max_{i=1}^{n} |\alpha_{i}| < \min_{i=1}^{n} |\lambda_{i}|$. Therefore, since
  $\lambda_{i} + \alpha_{i} > 0$ for every $i$, the point
  $(\alpha_{0}, \lambda_{1} + \alpha_{1}, \dots, \lambda_{n} + \alpha_{n})$ lies in $L \cap A$,
  implying that this intersection is a non-degenerate segment.
\end{proof}

\begin{lemma}\label{l:linear_map_on_simplex}
  Let $A$ be a simplex and $f \from A \to \A^{m}$ be an affine map. Then exactly one of the
  following holds:
  \begin{enumerate}[nosep]
  \item $f$ is injective, and consequently an embedding.
  \item For any codimension-one face $B$ of $A$ we have $f(A) = f(\border{A} \setminus \inter{B})$.
  \end{enumerate}
  In the second case, $f(A)$ need not be a simplex.
\end{lemma}

\begin{proof}
  Let $n = \dim A$. We may assume that $A$ is embedded into the affine space $\A^{n}$. In this case,
  $\inter{A}$ and $\border{A}$ are the topological interior and boundary of $A$ in $\A^{n}$,
  respectively.

  If $f$ is injective, then $f(A) \neq f(\border{A} \setminus \inter{B})$ for any codimension-one
  face $B$ of $A$. Assume that $f$ is not injective. It suffices to show that
  $f(\inter{A} \cup \inter{B}) \subseteq f(\border{A} \setminus \inter{B})$.

  Take any $p \in \inter{A} \cup \inter{B}$ and consider the fiber $f^{-1}(f(p))$, which is an
  affine subspace of positive dimension in $\A^{n}$ containing $p$. By~\cref{l:lines_simplex}, every
  line $L$ in this subspace through $p$ intersects $A$ in a segment $xy$ with $x$ and $y$ in
  $\border{A}$. It suffices to show that at least one of $x$ and $y$ does not lie in $\inter{B}$.

  Assume, for contradiction, that both $x$ and $y$ lie in $\inter{B}$. Then $xy \subset \inter{B}$,
  so $p \in \inter{B}$, and the entire line $L$ lies in the hyperplane spanned by the face $B$.
  By~\cref{l:lines_simplex}, this implies $xy = L \cap B$ since $xy = L \cap A$. But then
  $x, y \in \border{B}$, contradicting the assumption that $x, y \in \inter{B}$.

  Thus, $f(p) = f(x) = f(y)$ belongs to $f(\border{A} \setminus \inter{B})$ and hence
  $f(\inter{A} \cup \inter{B}) \subseteq f(\border{A} \setminus \inter{B})$.
\end{proof}

\begin{definition}
  Let $P$ be a polyhedron. A triangulation $T$ of $P \times I$ is called \emph{cylindrical} if there
  exists a triangulation $T(P)$ of $P$ such that the projection $\pr_P \from P \times I \to P$ is
  simplicial with respect to $T$ and $T(P)$.
\end{definition}

For a simplex $A \in T(P)$, consider its full preimage under the projection $\pr_P$, that is, the
union of the preimages of $A$ and all its faces. This preimage triangulates the set $|A| \times I$;
we denote the corresponding subcomplex of $T$ by $T_A$.

We call a simplex $B \in T_A$ \emph{principal} if $\pr_P(B) = A$. Since
$\dim(A \times I) = \dim A + 1$, each principal simplex has dimension either $\dim A$ or
$\dim A + 1$. We call the $(\dim A + 1)$-dimensional principal simplices \emph{vertical} and the
$(\dim A)$-dimensional ones \emph{horizontal}. Denote by $\mathrm{Princ}(A)$ the set of principal
simplices of $T_A$.

\begin{lemma}\label{l:princ_orders}
  Let $p \in \inter{A}$. Then:
  \begin{enumerate}[nosep]
  \item For every horizontal simplex $B \in \mathrm{Princ}(A)$, the intersection
    $(\{p\} \times I) \cap B$ consists of a single point lying in $\inter{B}$.
  \item For every vertical simplex $B \in \mathrm{Princ}(A)$, the intersection $(\{p\} \times I) \cap B$
    is a non-degenerate segment, and $(\{p\} \times I) \cap \inter{B} \neq \varnothing$.
  \item The intersections $(\{p\} \times I) \cap \inter{B}$ for $B \in \mathrm{Princ}(A)$ form a
    decomposition of the segment $\{p\} \times I$ into points and intervals. Therefore, the natural
    order of points along $\{p\} \times I$ induces a linear order $\prec_p$ on $\mathrm{Princ}(A)$,
    where for distinct simplices $B_1, B_2 \in \mathrm{Princ}(A)$ we set
    \[
      B_1 <_p B_2 \quad \Longleftrightarrow \quad t_{1} < t_{2} \quad \text{for all } (p, t_{1}) \in
      \inter{B}_1 \cap (\{p\} \times I) \text{ and } (p, t_{2}) \in \inter{B}_2 \cap (\{p\} \times I).
    \]
    With respect to $<_{p}$, the vertical and horizontal simplices in $\mathrm{Princ}(A)$ alternate,
    and every horizontal simplex is a face of a simplex immediately larger and immediately smaller
    than it, whenever such simplices exist.
  \item The order $\prec_p$ is independent of the choice of the point $p \in \inter{A}$.
  \end{enumerate}
\end{lemma}

\begin{proof} Refer to~\cref{fig:order_on_simplexes}.
  \begin{enumerate}
  \item Since $\pr_{P}(B) = A$ and $\dim A = \dim B$, the restriction $\pr_{P}\!\rest{B}$ is an affine
    homeomorphism between $B$ and $A$.
  \item Consider the affine map $\pr_{P}\!\rest{B} \from B \to A$. Since $\dim B = \dim A + 1$, the
    fiber $\pr_{P}^{-1}\!\rest{B}(p) = (\{p\} \times I) \cap B$ is the intersection of a
    one-dimensional affine subspace with $B$, which must be a non-degenerate segment $xy$
    by~\cref{l:lines_simplex}.

    We claim that $xy \cap \inter{B} \neq \varnothing$. Suppose, for contradiction, that
    $xy \subset \border{B}$. Then both the endpoints $x$ and $y$ must lie in a codimension-one face
    $C$ of $B$. Since $p$ lies in the interior of $A$, we have $f(C) = A$. Thus, $C$ is a horizontal
    principal simplex. From the previous point, it follows that $(p \times I) \cap C$ is a single
    point lying in $\inter{C}$. Thus, $(\{p\} \times I) \cap B = \{ x = y \in \inter{C} \}$, which
    contradicts~\cref{l:lines_simplex}.
  \item It suffices to show that every point $(p, t)$ in $\{p\} \times I$ lies in the interior of a
    principal simplex. Let $B \in T_A$ be the simplex of minimal dimension containing $(p, t)$.
    Since $p$ lies in the interior of $A$, we have $f(B) = A$, so $B \in \mathrm{Princ}(A)$. If
    $\dim B = \dim A$, the simplex $B$ is horizontal, and $\pr_{B} \from B \to A$ is a
    homeomorphism, so $(p, t) \in \inter{B}$. Otherwise we have $\dim B = \dim A + 1$ and
    $(p, t) \in \inter{B}$ because $p$ does not lie in any proper face of $B$, by the minimality of
    the dimension of $B$.
  \item Fix $p \in \inter{A}$ and let $\mathrm{Princ}(A) = \{ B_1, \dots, B_{2k+1} \}$ with
    $B_1 \prec_p B_2 \prec_p \dots \prec_p B_{2k+1}$. For $p' \in \inter{A}$, let $\gamma$ be a path
    in $\inter{A}$ from $p$ to $p'$. For each horizontal simplex $B_{2i+1}$, lift $\gamma$ along the
    projection $\pr_P\!\rest{\inter{B}_{2i+1}}$ to obtain a path $\gamma_{2i+1}$ in
    $\inter{B}_{2i+1}$. For each vertical simplex $B_{2i}$, define a path
    $t \mapsto \gamma_{2i}(t) = \frac{1}{2}(\gamma_{2i-1}(t) + \gamma_{2i+1}(t))$ in
    $\inter{B}_{2i}$. Clearly, $\gamma_{2i}$ is a lifting of $\gamma$ along
    $\pr_P\!\rest{\inter{B}_{2i}}$.

    The paths $\gamma_i$ are disjoint and project to $\gamma$ under $\pr_P$. The map
    $t \mapsto (\pr_I \gamma_1(t), \dots, \pr_I \gamma_{2k+1}(t))$ thus defines a path in the
    ordered configuration space of $2k+1$ points in $I$. Since the order of the points must be
    preserved along the path,\footnote{The ordered configuration space of $2k+1$ points in $I$
      deformation retracts onto the symmetric group $S_{2k+1}$.} we conclude
    $B_1 \prec_{p'} B_2 \prec_{p'} \dots \prec_{p'} B_{2k+1}$. \qedhere
  \end{enumerate}
\end{proof}

Before proving the theorem on cylinderwise collapsing, let us make the following simple
observations:
\begin{enumerate}
\item Every simplex $B$ of $T$ is principal in a unique subcomplex $T_{A}$ with $A = f(B)$. Hence,
  given the linear orders $\prec_{A}$ on $\mathrm{Princ}(A)$ for all $A \in T(P)$ as defined
  in~\cref{l:princ_orders}, we can define a partial order
  \[
    \prec\ \defeq\bigsqcup_{A \in T(P)} \prec_{A}
  \]
  on the simplices of $T$ by declaring $B_{1} \prec B_{2}$ whenever $f(B_{1}) = f(B_{2})$ and
  $B_{1} \prec_{f(B_{1})} B_{2}$.
\item For every vertical simplex $B \in \mathrm{Princ}(A)$, there exist precisely two horizontal
  simplices $C, C' \in \mathrm{Princ}(A)$ that are codimension-one faces of $B$, lying respectively
  immediately below and above $B$ with respect to $\prec_{A}$.
\item A subcomplex $S \subseteq T$ is downward closed in $(T, \prec)$ if and only if the underlying
  subpolyhedron $|S| \subseteq P \times I$ is downward closed.
\end{enumerate}

\begin{figure}[tb]
  \begin{tikzpicture}[scale=2,line cap=round]
    \fill[fill=gray!25] (1.5,0) -- (1.5,1.5+2.4) -- (0, 1.5+2.4) -- (0, 0) -- cycle;
    \draw (1.5,0) coordinate (A0)
          -- ++(-1.5,0) coordinate (A1)
          -- ++(1.5,0.75) coordinate (A2)
          -- ++(-1.5,0.6) coordinate (A3)
          -- ++(1.5,0.6) coordinate (A4)
          -- ++(-1.5,0.6) coordinate (A5)
          -- ++(1.5,0.6) coordinate (A6)
          -- ++(-1.5,0.75) coordinate (A7)
          -- ++(1.5,0) coordinate (A8);
    \draw (A0) -- (A8) -- (A7) -- (A1) -- cycle;

    \foreach \i in {0,...,8} {
      \pgfmathparse{mod(\i,2)==0 ? "right" : "left"}
      \edef\place{\pgfmathresult}
      \node[\place] at (A\i) {$a_{\i}$};
    }

    \foreach \i in {0,...,7} {
      \pgfmathparse{mod(\i,2)==0 ? 0.6 : 0.4}
      \xdef\pos{\pgfmathresult}
      
      \path (A\i) -- (A\the\numexpr\i+1\relax) coordinate[pos=\pos] (B\i);
      \fill[blue] (B\i) circle (1pt);
    }
    \draw[blue,thick] (B0) -- (B7);
    \node[below] at (B0) {$(p, 0)$};
    \node[above] at (B7) {$(p, 1)$};
  \end{tikzpicture}
  \centering
  \caption{The blue segment $\{p\} \times I$ intersects the principal horizontal simplices
    $a_i a_{i+1}$ in points and the principal vertical simplices $a_i a_{i+1} a_{i+2}$ along
    segments.}\label{fig:order_on_simplexes}
\end{figure}

\begin{lemma}[{Cylinderwise collapsing~(Corollary~2 to Lemma~45 in~\cite{zeeman_seminar})}]\label{theorem:cylindrical_collapse}
  Let $T$ be a cylindrical triangulation of $P \times I$. Suppose $T_{1}$ and $T_{2}$ are
  subcomplexes of $T$ such that:
  \begin{enumerate}[nosep]
  \item $P \times \{ 0 \} \subseteq |T_{2}| \subseteq |T_{1}|$,
  \item both $T_{1}$ and $T_{2}$ are downward closed in $(T, \prec)$.
  \end{enumerate}

  Then $T_{1}$ collapses onto $T_{2}$ through a sequence of elementary collapses
  $T' \collapse{B}{A} T''$. In each collapse, $B$ is a horizontal simplex which is maximal with
  respect to the order $\prec$ restricted to the current subcomplex $T'$, while $A$ is a vertical
  simplex immediately smaller than $B$ such that $f(A) = f(B)$. In particular, every intermediate
  subcomplex in this sequence remains downward closed.
\end{lemma}

\begin{proof}
  We argue by induction on the number of simplices in $T_{1} \setminus T_{2}$. If $T_{1} = T_{2}$,
  there is nothing to prove. Suppose $T_{1} \neq T_{2}$.

  Consider the simplices in $T_{1} \setminus T_{2}$ of highest dimension. Let $B$ be one of these
  simplices that is maximal among them with respect to the order $\prec$. Define $A \defeq f(B)$.
  Note that $B$ does not lie in $P \times \{0\}$, since $P \times \{0\} \subseteq |T_{2}|$ and
  $B \notin T_{2}$. Hence $B$ is not minimal in $T_{A}$.

  We claim that $B$ is vertical. Indeed, if it is horizontal, then there would exist a vertical
  simplex $B' \in \mathrm{Princ}(A)$ immediately below $B$ with $\dim B' = \dim B + 1$, which would
  necessarily lie in $T_{1}$, contradicting the maximality of the dimension of $B$.

  Let $C$ be the face of $B$ immediately larger than $B$ in $\mathrm{Princ}(A)$. Since
  $B \notin T_{2}$ and $T_{2}$ is downward closed, we have $C \notin T_{2}$. It follows that $C$ is
  the maximal horizontal simplex in $\mathrm{Princ}(A)$ contained in $T_{1}$. Indeed, if there
  exists a horizontal simplex $C' \succ_{A} C$ in $T_{1}$, then all vertical simplices between $C'$
  and $C$ would also belong to $T_{1}$, because $T_{1}$ is downward closed. This would contradict
  the choice of $B$.

  We now show that $T_{1} \collapse{C}{B} T_{1} \setminus \{C, B\}$. It suffices to prove that $C$
  is a free face of $B$ in $T_{1}$.

  Assume for contradiction that there exists another simplex $B'$ in $T_{1}$, distinct from $B$,
  having $C$ as a face. Then $B'$ cannot lie in $\mathrm{Princ}(A)$. Indeed, otherwise we would have
  $B' \succ_{A} B$ and $B' \in T_{1} \setminus T_{2}$, contradicting the choice of $B$. It follows
  that $A' \defeq f(B')$ is a simplex having $A$ as a face. Moreover, the dimension of $A'$ is
  $\dim A + 1 = \dim B'$, and hence $B'$ is a horizontal simplex in $\mathrm{Princ}(A')$.

  Since $C$ does not lie in $P \times \{0\}$, neither does $B'$. Therefore, there exists a vertical
  $(\dim B' + 1)$-simplex immediately below $B'$ in $\mathrm{Princ}(A')$. This simplex lies in
  $T_{1}$ since $T_{1}$ is downward closed. Moreover, it has dimension $\dim B + 1$ and cannot not
  lie in $T_{2}$, as otherwise its faces $B'$ and $C$ would lie in $T_{2}$. This contradicts the
  choice of $B$. Therefore, $C$ is a free face of $B$.

  Consequently, $T_{1} \collapse{C}{B} T_{1} \setminus \{C, B\}$, and the induction hypothesis
  applies to $T_{1} \setminus \{C, B\}$, completing the proof.
\end{proof}

Finally, we state a simple lemma that will be used in the next section.

\begin{lemma}\label{l:triangulate_base}
  Let $T$ be a cylindrical triangulation of $P \times I$. Then the subpolyhedra $P \times \{ 0 \}$
  and $P \times \{ 1 \}$ are triangulated by subcomplexes of $T$.
\end{lemma}

\begin{proof}
  Let $T(P)$ be the triangulation of $P$ such that $\pr_{P} \from T \to T(P)$ is simplicial. We show
  that $P \times \{ 0 \}$ is triangulated by a subcomplex; the proof for $P \times \{ 1 \}$ is the
  same. Pick a point $(x, 0) \in P \times \{ 0 \}$, and choose a simplex $A \in T(P)$ such that
  $x \in \inter{A}$. By~\cref{l:princ_orders}, there exists a simplex $B \in \mathrm{Princ}(A)$ with
  $(x, 0) \in \inter{B}$. Clearly, $B$ must be the minimal simplex in $\mathrm{Princ}(A)$. Since $B$
  is minimal, we have $|B| \subseteq P \times \{ 0 \}$. Hence, the collection of all minimal
  simplices forms a subcomplex that triangulates $P \times \{ 0 \}$.
\end{proof}

\section{Collapsibility and free deformation retraction}

In this section, we prove the main result of this text:

\begin{theorem}\label{theorem:main_theorem}
  Let $P$ be a compact polyhedron and $Q \subseteq P$ a compact subpolyhedron. Then $P$ collapses to
  $Q$ if and only if $P$ freely deformation retracts onto $Q$.
\end{theorem}

\begin{proof}[Proof that if $P \collapsible Q$, then $P$ freely deformation retracts onto $Q$.]
  By~\cref{lemma:join_two_retractions} and the definition of collapsibility, it suffices to show
  that whenever we have an elementary collapse $K \collapse{B}{A} K \setminus \{B, A\}$, the
  realization $|K|$ freely deformation retracts onto $|K \setminus \{B, A\}|$. We construct such a
  retraction $h \from |K| \times I \to |K|$ explicitly. Since $h$ must be a strong deformation
  retraction, it should fix $K \setminus (\inter{A} \sqcup \inter{B})$. Thus, it remains to define
  $h$ on $(\inter{A} \sqcup \inter{B}) \times I$.

  Let us first describe informally how $h_{t} \defeq h(\cdot, t)$ acts on
  $\inter{A} \sqcup \inter{B}$ for $t > 0$. Without loss of generality, assume that in the
  barycentric coordinate system of $A$ the face $B$ is given by $x_1 = 0$, and let $n = \dim A$.
  Consider the plane $P = \{ x_1 = t \}$ with $0 < t < 1$, parallel to $B$. This plane divides $A$
  into two regions: one containing the vertex $a = (1, 0, \dots, 0)$ opposite to $B$, and the other
  containing $B$ itself. On the region containing $a$, we set $h_{t}(x) = x$. On the other region,
  $h_{t}$ projects each point $x \in \inter{A} \sqcup \inter{B}$ along the line orthogonal to $B$
  until it meets either the boundary $\partial A \setminus \inter{B}$ or the plane $P$;
  see~\cref{fig:definiton_h}.

  \begin{figure}[tb]
    \centering
    \begin{subfigure}{0.49\textwidth}
      \tdplotsetmaincoords{100}{100}
      \begin{tikzpicture}[scale=2.5,tdplot_main_coords]
        \coordinate (A) at (1,-1,-1);
        \coordinate (B) at (1,1,1);
        \coordinate (C) at (-1,1,-1);
        \coordinate (D) at (-1,-1,1);

        \coordinate (BP) at ($0.45*(A) + 0.55*(B)$);
        \coordinate (CP) at ($0.45*(A) + 0.55*(C)$);
        \coordinate (DP) at ($0.45*(A) + 0.55*(D)$);

        \draw[black!80,dashed] (C)--(D);
        

        \filldraw[fill=gray!50,opacity=0.4] (A)--(BP)--(CP)--cycle;
        \filldraw[fill=gray!50,opacity=0.4] (A)--(BP)--(DP)--cycle;

        \filldraw[fill=red!30,opacity=0.9] (BP)--(CP)--(DP)--cycle;
        \filldraw[fill=red!30,opacity=0.3] (BP)--(CP)--(C)--(B)--cycle;
        \filldraw[fill=red!30,opacity=0.3] (BP)--(DP)--(D)--(B)--cycle;

        \draw[black] (A)--(B)--(D)--cycle;
        \draw[black] (A)--(B)--(C)--cycle;

        \coordinate (P) at ($0.3333*(BP)+0.3333*(CP)+0.3333*(DP)$);
        \node at (P) {$P$};

        \coordinate (x) at ($0.2*(B)+0.5*(C)+0.3*(D)$);
        \coordinate (hx) at ($0.45*(A)+0.05*(B)+0.35*(C)+0.15*(D)$);
        \fill[blue] (x) circle (1pt) node[right] {$x$};
        \fill[blue] (hx) circle (1pt) node[below left] {$h_{t}(x)$};
        \draw[->, blue, shorten >=3pt] (x) -- (hx);

        \coordinate (y) at ($0.48*(B)+0.1*(C)+0.42*(D)$);
        \coordinate (hy) at ($0.3*(A)+0.38*(B)+0*(C)+0.32*(D)$);
        \fill[blue] (y) circle (1pt) node[right] {$y$};
        \fill[blue] (hy) circle (1pt) node[below left] {$h_{t}(y)$};
        \draw[->, blue, shorten >=3pt] (y) -- (hy);

        \node at ($0.33*(B) + 0.33*(C) + 0.33*(D)$) {$B$};
        
      \end{tikzpicture}
    \end{subfigure}
    \begin{subfigure}{0.49\textwidth}
      \begin{tikzpicture}[scale=3,line cap=round]
        \coordinate (B) at (60:1);
        \coordinate (C) at (-60:1);
        \coordinate (A) at (180:1);

        \node[below left] at (A) {$a$};

        \coordinate (BP) at ($0.45*(A) + 0.55*(B)$);
        \coordinate (CP) at ($0.45*(A) + 0.55*(C)$);
        
        \filldraw[draw=black,fill=gray!20] (BP)--(A)--(CP);
        \draw[draw=red,line width=0.5mm] (BP)--(B);
        \draw[draw=red,line width=0.5mm] (CP)--(C);
        \draw[draw=red,line width=0.5mm] (BP)--(CP);
        \draw[dashed] (B)--(C);
        \node[right] at ($(B)!0.5!(C)$) {$B$};
        \node[right] at ($(BP)!0.5!(CP)$) {$P$};

        \coordinate (x) at ($0.35*(B)+0.65*(C)$);
        \coordinate (hx) at ($0.45*(A)+0.125*(B)+0.425*(C)$);
        \fill[blue] (x) circle (1pt) node[right] {$x$};
        \fill[blue] (hx) circle (1pt) node[above left] {$h_{t}(x)$};
        \draw[->, blue, dashed, shorten >=3pt] (x) -- (hx);

        \coordinate (y) at ($0.85*(B)+0.15*(C)$);
        \coordinate (hy) at ($0.3*(A)+0.7*(B)+0*(C)$);
        \fill[blue] (y) circle (1pt) node[right] {$y$};
        \fill[blue] (hy) circle (1pt) node[above left] {$h_{t}(y)$};
        \draw[->, blue, dashed, shorten >=3pt] (y) -- (hy);
      \end{tikzpicture}
    \end{subfigure}
    
    \caption{The action of $h_{0.45}$ on a 3-simplex (left) and on a 2-simplex (right). The image of
      the region containing $B$ is shown in red.}\label{fig:definiton_h}
  \end{figure}

  In barycentric coordinates, assuming $B$ is defined by $x_1 = 0$, we can write
  \[
    h((x_{1}, \dots, x_{n+1}), t) = \begin{cases}
      (x_1, \dots, x_{n+1}), &x_1 \geq t, \\
      \left(t, x_2 - \frac1n(t - x_1), \dots, x_{n+1} - \frac1n(t - x_1)\right), &x_1 \leq t \text{ and } x_{m} \geq \frac1n(t - x_1), \\
      (x_{1} + nx_m, x_2 - x_m, \dots, x_{n+1} - x_m), &x_1 \leq t \text{ and } x_m \leq \frac1n(t -
                                                         x_1),
    \end{cases}
  \]
  where $x_m \defeq \min(x_2, \dots, x_{n+1})$. We may assume that $A$ is embedded in $\A^{n}$ and
  hence $A \times I$ in $\A^{n+1}$. The defining conditions of $h$ then partition $\A^{n+1}$ into
  convex regions on each of which $h$ is affine, hence $h$ is piecewise linear. Is is
  straightforward to check, and it is also clear from the informal description provided above, that
  $h$ is free.
\end{proof}

We prove the converse implication through a sequence of lemmas. First, we introduce some additional
notation and auxiliary constructions.

\begin{figure}[tb]
  \centering
  \begin{overpic}[width=0.4\textwidth]{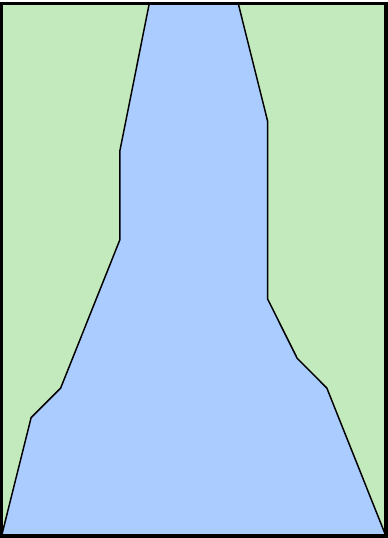}
    \put(38,25){\makebox(0,0){$M \defeq H(P \times I)$}}
    \put(12,80){\makebox(0,0){$N$}}
  \end{overpic}
  \caption{The sets $M$ and $N$ inside $P \times I$.}\label{fig:M_N}
\end{figure}

Let $h$ be a piecewise-linear free deformation retraction of $P$ onto $Q \subseteq P$. Define
$H \from P \times I \to P \times I$ as $H(x,t) \defeq (h(x,t), t)$, so that $H = h \times \pr_I$.
Clearly, $H$ is piecewise-linear, and hence the subset $M \defeq H(P \times I)$ is a compact
subpolyhedron of $P \times I$; see~\cref{fig:M_N}.

\begin{lemma}\label{l:MF_prop}
  The following holds:
  \begin{enumerate}[nosep]
  \item\label{l:M_incl} $P \times \{ 0 \} \subset M$, $Q \times I \subseteq M$, and
    $M \cap (P \times \{ 1 \}) = Q \times \{ 1 \}$.
  \item\label{l:M_downclosed} Both $M$ and its topological interior $\tinter M$ taken in
    $P \times I$ are downward closed.
  \item\label{l:F_prop} $H$ is a level-preserving retraction of $P \times I$ onto $M$ such that for
    every $(x, t) \in \tinter M$ we have $H^{-1}((x, t)) = \{(x, t)\}$.
  \end{enumerate}
\end{lemma}

\begin{proof}
  Statement~(\ref{l:M_incl}) follows directly from the definition of $M$.

  To prove~(\ref{l:M_downclosed}), we first show that if $(x, t) \in M$, then
  $\{ x \} \times [0, t] \subseteq M$. For any $s \leq t$, since $(x, t) \in M$, there exists
  $y \in P$ such that $h(y, t) = x$. Then
  \[
    H(x, s) = (h(x, s), s) = (h(h(y, t), s), s) = (h(y, t), s) = (x, s) \in M.
  \]
  Now suppose $(x, t) \in \tinter M$. Then there exists an open neighborhood of the form
  $U_x \times (t - \delta, t + \delta) \subseteq M$ for some $\delta > 0$ and an open neighborhood
  $U_x$ of $x$ in $P$. Since $M$ is downward closed, for every $(x, s)$ with $0 \leq s \leq t$ the
  open neighborhood $U_{x} \times ([0, 1] \cap (s - \delta, s + \delta))$ is contained in $M$,
  implying $(x, s) \in \tinter M$. Hence $\tinter M$ is downward closed.
 
  By construction, $H$ preserves levels: for each $t$, we have
  $H(P \times \{ t \}) \subseteq P \times \{ t \}$. Moreover,
  \[
    H(H(x, t)) = (h(h(x, t), t), t) = (h(x, t), t) = H(x, t),
  \]
  hence $H$ is a retraction.

  To prove that $H^{-1}(x, t) = \{(x, t)\}$ for every $(x, t) \in \tinter M$, we argue by
  contradiction. Suppose that there exists $(x, t) \in \tinter M$ such that $|H^{-1}(x, t)| > 1$;
  equivalently, there exists $y \neq x$ with $h(y, t) = x$. Define
  $\hat{t} \defeq \inf \{ s \in [0, 1] \mid h(y, s) = x \}$. By continuity of $h$, we obtain
  $h(y, \hat{t}) = x$. Since $h(y, 0) = y$, it follows that $0 < \hat{t} \leq t$. As $\tinter M$ is
  downward closed, we have $(x, \hat{t}) \in \tinter M$. Let
  $U \defeq U_x \times (\hat{t} - \epsilon, \hat{t} + \epsilon) \subset \tinter M$ be an open
  neighborhood of $(x, \hat{t})$, where $U_{x}$ is a neighborhood of $x$ in $P$.

  Consider the curve $\psi \from [0, \hat{t}] \to P$ given by $\psi(s) = h(y, s)$. Since
  $\psi(\hat{t}) = x$, the preimage $\psi^{-1}(U_{x})$ contains a segment $[t_0, \hat{t}]$ with
  $t_0 < \hat{t}$. Set $z \defeq \psi(t_{0}) \in U_{x}$. As $t_{0} < \hat{t}$, we have $z \neq x$.
  Moreover, $(z, \hat{t}) \in U \subset M$, and hence $H(z, \hat{t}) = (z, \hat{t})$ because $H$ is
  a retraction onto $M$. On the other hand,
  \[
    h(z, \hat{t}) = h(h(y, t_0), \hat{t}) = h(y, \hat{t}) = x,
  \]
  so $H(z, \hat{t}) = (x, \hat{t}) \neq (z, \hat{t})$, a contradiction. Therefore,
  $H^{-1}(x, t) = \{ (x, t) \}$ for every $(x, t) \in \tinter M$.
\end{proof}

\begin{figure}[tb]
  \centering
  \begin{subfigure}{0.53\textwidth}
    \centering
    \raisebox{-0.5\height}{\begin{overpic}[width=\textwidth]{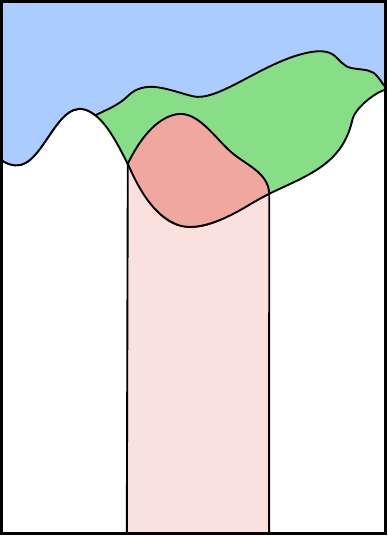}
      \put(35,92){\makebox(0,0){$Z$}}
      \put(52,78){\makebox(0,0){$Y \setminus Z$}}
      \put(37,66){\makebox(0,0){$Y \setminus h^{-1}(h(Z))$}}
      \put(37,40){\makebox(0,0){$S_{-}(Y \setminus h^{-1}(h(Z)))$}}
    \end{overpic}}
  \end{subfigure}
  \begin{subfigure}{0.45\textwidth}
    \centering
    \raisebox{-0.5\height}{\begin{overpic}[width=0.85\textwidth]{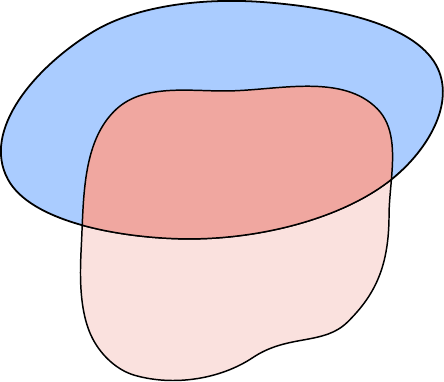}
      \put(52,76){\makebox(0,0){$h(Z)$}}
      \put(52,50){\makebox(0,0){$h(Y) \setminus h(Z)$}}
      \put(52,22){\makebox(0,0){$h(S_{-}(Y \setminus h^{-1}(h(Z))))$}}
    \end{overpic}}
  \end{subfigure}
  \caption{Illustration for~\cref{l:path_order}. The left picture shows $X \times I$ with the sets
    $Z \subset Y$, $Y \setminus h^{-1}(h(Z))$, and the downward closure of
    $Y \setminus h^{-1}(h(Z))$. The right picture shows the images of these sets under $h$. The
    lemma claims that the pink set in the image does not intersect the blue set representing
    $h(Z)$.}\label{fig:path_order}
\end{figure}

\begin{lemma}\label{l:path_order}
  Suppose $h \from X \times I \to X$ is a free deformation retraction (not necessarily
  piecewise-linear), and $Z \subseteq Y \subseteq X \times I$ are upward closed subsets. Then
  \[
    h(S_{-}(Y \setminus h^{-1}(h(Z))) \cap h(Z) = \varnothing,
  \]
  where $S_{-}(\cdot)$ denotes the downward closure operation in $X \times I$.
\end{lemma}

Since the statement of the lemma may seem somewhat intimidating, let us unpack its meaning; see
also~\cref{fig:path_order} for an illustration. Fix a point $x \in X$ and write
$\phi(t) \defeq h(x, t)$ for its trajectory under $h$. Because $Y$ and $Z$ are upward closed, their
intersections with $\{ x \} \times I$ form upper sets in $I$, that is, each intersection is either a
half-open interval $\{ x \} \times (s, 1]$ or a closed subsegment $\{ x \} \times [s, 1]$. The lemma
then states the following. If for some $t \in I$ the point $(x, t)$ belongs to
$Y \setminus h^{-1}(h(Z))$, that is, $(x, t) \in Y$ but $\phi(t) \notin h(Z)$, then the entire
initial part of the trajectory $\phi([0, t]) = \{ h(x, s) \mid 0 \leq s \leq t \}$ also avoids
$h(Z)$.

\begin{proof}[Proof of~\cref{l:path_order}]
  Let $(x, t) \in Y \setminus h^{-1}(h(Z))$, that is, $(x, t) \in Y$ and $h(x, t) \notin h(Z)$. We
  need to show that $h(x \times [0, t]) \cap h(Z) = \varnothing$. Assume, for contradiction, that
  there exists $s \in [0, t]$ such that $h(x, s) = h(y, r)$ for some $(y, r) \in Z$. Since $h$ is a
  free deformation retraction, we have
  \[
    h(x, t) = h(h(x, s), t) = h(h(y, r), t) = \begin{cases} h(y, r) &r \geq t, \\ h(y, t) &r \leq t. \end{cases}
  \]
  In both cases, we have $h(x, t) \in h(Z)$ because $(y, r) \in Z$, and $Z$ is upward closed so
  $(y, t) \in Z$ whenever $r \leq t$. Therefore, we conclude that
  $h(x \times [0, t]) \cap h(Z) = \varnothing$. Since this holds for every
  $(x, t) \in Y \setminus h^{-1}(h(Z))$, the lemma follows.
\end{proof}

\begin{lemma}\label{l:top_interior}
  Let $S \subseteq T$ be simplicial complexes. For a simplex $A \in T$, let $F(A)$ denote the set of
  all simplices of $T$ that have $A$ as a face; in particular, $A \in F(A)$. Then
  \[
    \tinter_{|T|} |S| = \bigsqcup \{ \inter{A} \mid F(A) \subseteq S \}.
  \]
\end{lemma}

\begin{proof}
  Take $x \in \tinter_{|T|} |S|$, and let $A \in T$ be the unique simplex with $x \in \inter{A}$.
  Choose an open neighborhood $U_x \subset |S|$ of $x$ in $|T|$. For each $B \in F(A)$, we have
  $x \in B$, so $U_x \cap B$ is a nonempty open subset of $B$. Since $\inter{B}$ is dense in $B$, it
  follows that $U_x \cap \inter{B} \neq \varnothing$. Pick $y \in U_x \cap \inter{B}$. Then $B$ is
  the unique simplex of $T$ containing $y$ in its relative interior. But $y \in U_x \subset |S|$, so
  $B \in S$. As this holds for all $B \in F(A)$, we obtain $F(A) \subseteq S$, and therefore
  $x \in \bigcup \{ \inter{A} \mid F(A) \subseteq S \}$.

  Conversely, let $x \in \inter{A}$ with $F(A) \subseteq S$. Set $C(A) \defeq T \setminus F(A)$. If
  $B \in C(A)$ and $D$ is a face of $B$, then $A$ is not a face of $D$, so $D \in C(A)$. Hence
  $C(A)$ is a subcomplex of $T$. Therefore $|C(A)|$ is a closed subpolyhedron of $|T|$ not
  containing $x$, and $|T| \setminus |C(A)|$ is an open neighborhood of $x$ in $|T|$. Since
  $|C(A)| = \bigcup_{B \in C(A)} |B|$, a simplex $B \in T$ meets $|T| \setminus |C(A)|$ only if
  $B \notin C(A)$, that is, $B \in F(A)$. Therefore, $|T| \setminus |C(A)| \subseteq |S|$, and we
  conclude that $|T| \setminus |C(A)| \subseteq |S|$, so $x \in \tinter_{|T|} |S|$.

  Since relative interiors of distinct simplices are disjoint, the union above is disjoint,
  which completes the proof.
\end{proof}

\begin{lemma}\label{l:collapse_image}
  Let $S \subset K$ be simplicial complexes, and let $f \from |K| \to P$ be a piecewise-linear map
  to a polyhedron $P$ which is affine on the simplices of $K$ with respect to a triangulation $T(P)$
  of $P$; that is, for every simplex $A \in K$ there exists a simplex $B \in T(P)$ such that
  $f(|A|) \subseteq B$ and $f\rest{|A|}$ is affine.

  Suppose that $f(|S|) \cap f(|K| \setminus |S|) = \varnothing$ and $f\rest{|K| \setminus |S|}$ is
  injective. Then, if $K$ simplicially collapses to $S$, the polyhedron $f(|K|)$ collapses to
  $f(|S|)$.
\end{lemma}

\begin{proof}
  First, consider the case of a single elementary collapse $K \collapse{B}{A} S$. Then
  $|K| = |S| \sqcup \inter{A} \sqcup \inter{B}$. By the assumptions,
  $f(\inter{A} \sqcup \inter{B}) \cap f(|S|) = \varnothing$. In particular, since
  $\border{A} \setminus \inter{B} \subseteq |S|$, we have
  $f(A) = f(\inter{A} \sqcup \inter{B}) \sqcup f(\border{A} \setminus \inter{B}) \neq f(\border{A}
  \setminus \inter{B})$. By~\cref{l:linear_map_on_simplex}, $f\rest{A}$ is an affine embedding of
  $A$ into a simplex in $T(P)$. Therefore, $f(A)$ is a piecewise-linear $(\dim A)$-ball, and
  $f(A) \cap f(|S|) = f(\border{A} \setminus \inter{B})$ is a piecewise-linear $(\dim A - 1)$-ball.
  It follows that $f(|K|) = f(|S|) \cup f(A)$ collapses polyhedrally to $f(|S|)$.

  Now suppose $K$ collapses to $S$ by $n > 1$ elementary collapses. Let the first elementary
  collapse be $K \collapse{B}{A} S'$. Then we can write
  \[
    |K| = (\inter{B} \sqcup \inter{A}) \sqcup |S'| = \underbrace{(\inter{B} \sqcup \inter{A}) \sqcup (|S'|
    \setminus |S|)}_{|K| \setminus |S|} \sqcup |S|.
  \]

  Using the assumptions of the lemma, we obtain
  \[
  \begin{array}{lll}
    f(|K|) &= f((\inter{B} \sqcup \inter{A}) \sqcup (|S'| \setminus |S|) \sqcup |S|)    & \\[0.5em]
           &= f((\inter{B} \sqcup \inter{A}) \sqcup (|S'| \setminus |S|)) \sqcup f(|S|) &\quad\text{(since $f(|S|) \cap f(|K| \setminus |S|) = \varnothing$)} \\[0.5em]
           &= f(\inter{B} \sqcup \inter{A}) \sqcup f(|S'| \setminus |S|) \sqcup f(|S|). &\quad\text{(since $f\rest{|K| \setminus |S|}$ is injective)}
  \end{array}
  \]

  Thus the hypotheses of the lemma apply to the pairs $S' \subset K$ and $S \subset S'$. Indeed,
  $f(|K| \setminus |S'|) = f(\inter{A} \sqcup \inter{B})$ is disjoint from $f(|S'|)$, and
  $f(|S'| \setminus |S|)$ is disjoint from $f(|S|)$. Moreover, the restrictions
  $f\rest{|K| \setminus |S'|}$ and $f\rest{|S'| \setminus |S|}$ are injective, since
  $f\rest{|K| \setminus |S|}$ is injective. By the induction hypothesis, $f(|K|)$ collapses to
  $f(|S'|)$, and $f(|S'|)$ collapses to $f(|S|)$. Consequently, $f(|K|)$ collapses to $f(|S|)$.
\end{proof}

\begin{proof}[Proof of~\cref{theorem:main_theorem}]
  Recall that $h$ is a piecewise-linear free deformation retraction of $P$ onto $Q \subseteq P$,
  $H \from P \times I \to P \times I$ is defined by $H(x,t) \defeq (h(x,t), t)$, and
  $M \defeq H(P \times I) \subseteq P \times I$.
  
  First, we choose triangulations of $P \times I$ and $P$ as follows:
  \begin{enumerate}[nosep]
  \item Triangulate $P \times I$ so that $M$ is triangulated by a subcomplex.
  \item Triangulate $P$, and subdivide the triangulation of $P \times I$ from the previous step so
    that $h$ becomes simplicial.
  \item Subdivide the triangulations of $P \times I$ and $P$ so that the projection
    $\pr_{P} \from P \times I \to P$ is simplicial.
  \end{enumerate}

  Let $T(P \times I)$ denote the final triangulation of $P \times I$. By construction,
  $T(P \times I)$ is cylindrical, and $h$ is affine on every simplex of $T(P \times I)$ in the sense
  required by~\cref{l:collapse_image}. Let $T(M)$ be the subcomplex of $T(P \times I)$ that
  triangulates $M$.

  Define
  \[
    N \defeq ((P \times I) \setminus \tinter M) \cup (P \times \{1\}).
  \]
  We first verify that $N$ is upward closed. Since $\tinter M$ is downward closed by
  \crefitem{l:MF_prop}{l:M_downclosed}, it follows from~\cref{lemma:shadows_properties} that its
  complement is upward closed. As $P \times \{ 1 \}$ is trivially upward closed, we conclude that $N$
  is upward closed.

  Next, we claim that $N$ is a subpolyhedron of $P \times I$, triangulated by a subcomplex of
  $T(P \times I)$. Because $T(P \times I)$ is cylindrical, the subpolyhedron $P \times \{ 1 \}$ is
  triangulated by a subcomplex $T(P \times \{ 1 \})$ by~\cref{l:triangulate_base}. Thus, it suffices
  to show that $(P \times I) \setminus \tinter M$ is triangulated by a subcomplex.

  Define
  \[
    T' := \{ A \in T(P \times I) \mid A \cap \tinter M = \varnothing \} \subseteq T(P \times I).
  \]
  Clearly, $|T'| \subseteq (P \times I) \setminus \tinter M$.

  For the reverse inclusion, let $(x, t) \notin \tinter M$, and let $A \in T(P \times I)$ be the
  unique simplex such that $(x, t) \in \inter A$. We claim that $A \cap \tinter M = \varnothing$,
  which implies $(x, t) \in |T'|$.

  Suppose, for contradiction, that $A \cap \tinter M \neq \varnothing$. Since $M$ is triangulated by
  a subcomplex of $T(P \times I)$, by~\cref{l:top_interior}, its topological interior is a union of
  relative interiors of simplices. Hence $A \cap \tinter M \neq \varnothing$ implies
  $\inter A \subseteq \tinter M$. In particular, $(x, t) \in \inter A \subseteq \tinter M$,
  contradicting our assumption.

  Therefore, $N$ is triangulated by a subcomplex of $T(P \times I)$, which we denote by $T(N)$.

  We now show that $h(P \times \{ 1 \}) = Q$ and $h(N) = P$. The first equality follows directly
  from the fact that $h$ is a deformation retraction. Since $P \times \{ 1 \} \subset N$, it
  suffices to show that for every $x \notin Q$ there exists $(x, t) \notin \tinter M$ with
  $h(x, t) = x$.

  Fix $x \notin Q$ and consider the segment $\{ x \} \times [0, 1]$. Because $h(x, 0) = x$, the
  point $(x, 0)$ lies in $M$, so the set $\{ s \in [0,1] \mid (x, s) \in M \}$ is nonempty.
  Moreover, since $h(x, 1) \in Q$ and $x \notin Q$, we have $h(x, 1) \neq x$, and hence
  $(x, 1) \notin M$. Therefore this set is a proper subset of $[0,1]$. Let
  $t \defeq \sup \{ s \mid (x, s) \in M \}$. Since $M$ is closed, $(x, t) \in M$, and we must have
  $(x, t) \in \partial M = M \setminus \tinter M$. By \crefitem{l:MF_prop}{l:F_prop}, it follows
  that $h(x, t) = x$. This proves $h(N) = P$.

  By~\cref{theorem:cylindrical_collapse}, the subcomplex $T(N)$ collapses cylinderwise to
  $T(P \times \{1\})$. Let
  \[
    T(N) = T_{0} \collapse{B_{0}}{A_{0}} T_{1} \collapse{B_{1}}{A_{1}} \dots
    \collapse{B_{n-1}}{A_{n-1}} T_{n} = T(P \times \{1\})
  \]
  be the corresponding sequence of elementary collapses. Since $h(|T_{0}|) = P$ and
  $h(|T_{n}|) = Q$, to show that $P \collapsible Q$ it suffices to prove that
  \[
    h(|T_{i}|) \collapsible h(|T_{i+1}|) \quad\text{for each elementary collapse } T_{i}
    \collapse{B_{i}}{A_{i}} T_{i+1}.
  \]

  We now prove this. Clearly, we may assume that $h(|T_{i}|) \neq h(|T_{i+1}|)$, since otherwise the
  claim is trivial. Our plan is to apply~\cref{l:collapse_image}; however, we cannot guarantee that
  $h(|T_{i}| \setminus |T_{i+1}|)$ is disjoint from $h(|T_{i+1}|)$, because $h$ is not simplicial
  but only affine on the simplices of $T(P \times I)$. The following construction resolves this
  difficulty.

  Let
  \[
    L \defeq |T_{i}| \setminus h^{-1}\left(h(|T_{i+1}|)\right).
  \]
  Informally, $L$ is the part of $|T_{i}|$ responsible for the change in the image during the collapse
  $T_{i} \collapse{B_{i}}{A_{i}} T_{i+1}$.

  \begin{figure}[tb]
    \begin{subfigure}{0.49\textwidth}
      \centering
      \begin{overpic}[height=6cm]{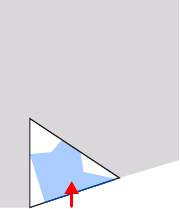}
        \put(30,18){\makebox(0,0){$L$}}
        \put(45,70){\makebox(0,0){$|T_{i+1}|$}}
      \end{overpic}
    \end{subfigure}%
    \begin{subfigure}{0.49\textwidth}
      \centering
      \begin{overpic}[height=6cm]{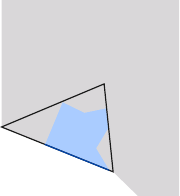}
        \put(40,30){\makebox(0,0){$h(L)$}}
        \put(40,80){\makebox(0,0){$h(|T_{i+1}|)$}}
      \end{overpic}
    \end{subfigure}
    \caption{A step in the proof of~\cref{theorem:main_theorem}. The left picture shows the simplex
      $A_{i}$ inside $|T_{i}|$. The collapse is from the bottom face toward the faces meeting the grey
      region $|T_{i+1}|$. The right picture shows the images of $A_{i}$ and of the set
      $L$.}\label{fig:collapsing}
  \end{figure}
  
  \begin{description}
  \item[\hypertarget{c:incl}{Claim 1}] $L \subseteq \inter{A}_{i} \sqcup \inter{B}_{i}$.
    \nopagebreak[3]
    
    Indeed, if $(x, t) \in L$ then by definition $h(x, t) \notin h(|T_{i+1}|)$. Hence $(x, t) \notin |T_{i+1}|$,
    and therefore
    \[
      (x, t) \in |T_{i}| \setminus |T_{i+1}| = \inter A_{i} \sqcup \inter B_{i}.
    \]
    
  \item[\hypertarget{c:identical_image}{Claim 2}] $h(|T_i| \setminus L) = h(|T_{i+1}|)$
    \nopagebreak[3]

    Since $|T_{i+1}| \subseteq |T_{i}| \setminus L$ by~\hyperlink{c:incl}{Claim~1}, we have
    $h(|T_{i+1}|) \subseteq h(|T_{i}| \setminus L)$. For the reverse inclusion, let
    $y \in h(|T_{i}| \setminus L)$. Then $y = h(x, t)$ for some $(x, t) \in |T_{i}| \setminus L$. By
    definition of $L$, the condition $(x, t) \notin L$ means that $h(x, t) \in h(|T_{i+1}|)$. Thus
    $y \in h(|T_{i+1}|)$, and the equality follows.
    
  \item[Claim 3] $|T_{i}| \setminus L$ is a compact subpolyhedron.
    \nopagebreak[3]

    Note that $|T_{i}| \setminus L = h^{-1}(h(|T_{i+1}|)) \cap |T_{i}|$. Indeed,
    by~\hyperlink{c:identical_image}{Claim~2}, we have
    $|T_i| \setminus L \subseteq h^{-1}(h(|T_{i+1}|)) \cap |T_i|$. Conversely, if
    $(x, t) \in h^{-1}(h(|T_{i+1}|)) \cap |T_i|$, then $h(x, t) \in h(|T_{i+1}|)$, so
    $(x, t) \notin L$ by the definition of $L$, and hence $(x, t) \in |T_i| \setminus L$. Thus, we
    have
    \[
      |T_{i}| \setminus L = h^{-1}(h(|T_{i+1}|)) \cap |T_{i}| = h\rest{|T_{i}|}^{-1}(h(|T_{i+1}|)).
    \]
    Since $h(|T_{i+1}|)$ is a compact subpolyhedron of $P$ and $h\rest{|T_{i}|}$ is a
    piecewise-linear map, this preimage is a compact subpolyhedron.
    
  \item[Claim 4] $h\rest{L}$ is injective.
    \nopagebreak[3]

    We have $|T_{i}| = |T_{i+1}| \cup A_{i}$, hence $h(|T_{i}|) = h(|T_{i+1}|) \cup h(A_{i})$. Since
    $h(|T_{i}|) \neq h(|T_{i+1}|)$, the set $h(A_{i})$ is not contained in $h(|T_{i+1}|)$. However,
    $\border{A}_{i} \setminus \inter{B}_{i} \subseteq |T_{i+1}|$. Thus
    $h(A_{i}) \neq h(\border{A}_{i} \setminus \inter{B}_{i})$. Because $h$ is affine on $A_{i}$, it
    follows from~\cref{l:linear_map_on_simplex} that $h\rest{A_{i}}$ is injective. Since
    $L \subseteq A_{i}$ by~\hyperlink{c:incl}{Claim~1}, the restriction $h\rest L$ is injective.
    
  \item[Claim 5] $|T_{i}| \setminus L$ is upward closed and $h(|T_{i}| \setminus L) \cap h(L) =
    \varnothing$.
    \nopagebreak[3]

    Recall that $L = |T_{i}| \setminus h^{-1}\left(h(|T_{i+1}|)\right)$. Applying
    \cref{l:path_order} to the inclusion $|T_{i+1}| \subset |T_{i}|$ and the map $h$ yields
    $h(S_{-}(L)) \cap h(|T_{i+1}|) = \varnothing$, where $S_{-}(L)$ denotes the downward closure of
    $L$.

    Thus $h(S_{-}(L) \cap |T_{i}|)$ is disjoint from $h(|T_{i+1}|)$, and consequently
    $h(S_{-}(L) \cap |T_{i}|)$ is a subset of $h(|T_{i}|) \setminus h(|T_{i+1}|)$. This implies
    $S_{-}(L) \cap |T_{i}| \subseteq L$. Since $L \subseteq S_{-}(L)$ and $L \subseteq |T_{i}|$, we
    obtain $L = S_{-}(L) \cap |T_{i}|$.

    Therefore,
    \[
      |T_{i}| \setminus L = |T_{i}| \setminus (S_{-}(L) \cap |T_{i}|) = |T_{i}| \setminus S_{-}(L) =
      |T_{i}| \cap ((P \times I) \setminus S_{-}(L)).
    \]
    Since $S_{-}(L)$ is downward closed, its complement in $P \times I$ is upward closed. Moreover,
    $|T_{i}|$ is upward closed because the collapsing is cylinderwise. Therefore, their intersection
    $|T_{i}| \setminus L$ is upward closed.

    Finally, by~\hyperlink{c:identical_image}{Claim 2} we have
    $h(|T_{i}| \setminus L) = h(|T_{i+1}|)$. Thus, $h(L) = h(S_{-}(L) \cap |T_{i}|)$ is disjoint
    from $h(|T_{i+1}|) = h(|T_{i}| \setminus L)$, completing the proof.
  \end{description}

  Now we are almost ready to apply~\cref{l:collapse_image}. To do so, we refine the triangulation
  $T(P \times I)$ in two stages. First, we subdivide it so that $|T_{i}| \setminus L$ is
  triangulated by a subcomplex. Second, we further subdivide the resulting triangulation so that the
  projection $\pr_{P}$ becomes simplicial with respect to some triangulation of $P$.

  Let $\hat{T}(P \times I)$ denote the final triangulation. In $\hat{T}(P \times I)$, the subcomplex
  $\hat{T}(|T_{i}|)$ triangulating $|T_{i}|$ collapses simplicially to the subcomplex
  $\hat{T}(|T_{i}| \setminus L)$ triangulating $|T_{i}| \setminus L$
  by~\cref{theorem:cylindrical_collapse}. Moreover,
  \begin{enumerate}[nosep]
  \item $h$ is affine on each simplex of $\hat{T}(P \times I)$, since it was already affine on
    $T(P \times I)$,
  \item $h(|T_{i}| \setminus L) \cap h(L) = \varnothing$,
  \item $h\rest L$ is injective.
  \end{enumerate}
  Therefore, by \cref{l:collapse_image}, the polyhedron $h(|T_{i}|)$ collapses to
  $h(|T_{i}| \setminus L) = h(|T_{i+1}|)$ (see~\cref{fig:collapsing}), completing the proof.
\end{proof}

\section{Injective metric and free deformation retraction}

In the present section, we discuss the connections between free deformation retraction and the
injective metrizability of spaces. In particular, we address the statement from~\cite{isbell64}
that an injective metric space is freely deformation retractable. We provide a counterexample to a
step in the proof given there, and we propose a correction for the case of compact spaces.
Finally, we discuss the potential use of these results for characterising collapsibility.

\begin{definition}
  A metric space $(X, d_{X})$ is called \emph{injective} if for every metric space $(Y, d_{Y})$,
  every subspace $Z \subseteq Y$, and every non-expansive map $f \from Z \to X$, meaning that
  $d_{X}(f(p), f(q)) \leq d_{Y}(p, q)$ for all $p, q \in Z$, there exists a non-expansive extension
  $\hat{f} \from Y \to X$ such that $d_{X}(\hat{f}(p), \hat{f}(q)) \leq d_{Y}(p, q)$ for all
  $p, q \in Y$ and $\hat{f}\rest{Z} = f$.
\end{definition}

There is an equivalent characterization of injective metric spaces: a metric space $(X, d_{X})$ is
injective if and only if for every collection of closed balls
$\{B(x_{\alpha}, r_{\alpha}) \mid \alpha \in I\}$ such that
$d_{X}(x_{\alpha}, x_{\beta}) \leq r_{\alpha} + r_{\beta}$ for all $\alpha, \beta \in I$, the
intersection $\bigcap_{\alpha \in I} B(x_{\alpha}, r_{\alpha})$ is non-empty. A metric space
satisfying this property is often called \emph{hyperconvex}; hence, hyperconvex and injective metric
spaces coincide. The reader may refer to~\cite{hyperconvex} for a proof of this equivalence and for
an overview of other properties of injective/hyperconvex spaces.

The following statement was made in~\cite{isbell64}; we refer to it as a conjecture, since, as we
discuss below, the proof given there contains a significant issue:

\begin{conjecture}[Theorem~1.1 in~\cite{isbell64}]
  An injective metric space is freely deformation contractible to each of its points.
\end{conjecture}

Below, we repeat the main line of the argument and point out the issue.

Suppose $(Y, d)$ is an injective metric space, and let $p \in Y$. Consider the collection
$\mathcal{X}$ of pairs $(X, h)$, where $X \subseteq Y$ and
$h \from X \times [0, +\infty) \to X$ is a map such that
\begin{enumerate}[nosep]
\item $h(X, 0) = \{ p \}$,
\item $h(X, t) \subseteq B(p, t)$ and $h(u, t) = u$ for all $u \in X \cap B(p, t)$,
\item $h$ is free: $h(h(u, t), s) = h(u, \min(t, s))$,
\item $h$ is non-expansive: $d(h(u, t), h(w, s)) \leq \max\{ d(u, w), |t-s| \}$.
\end{enumerate}

Define a partial order on $\mathcal{X}$ by declaring $(X_{1}, h_{1}) \prec (X_{2}, h_{2})$ whenever
$X_{1} \subseteq X_{2}$ and $h_{2}\rest{X_{1} \times [0, +\infty)} = h_{1}$. It is
straightforward to see that every chain in $\mathcal{X}$ is bounded above by the pair consisting of
the union of the sets in the chain and the correspondingly defined map on this union. Thus, by
Zorn's lemma, $\mathcal{X}$ has a maximal element; denote it by $(Z, h)$.

Since for every $(X, h)$ we have $(X, h) \preceq (\overline{X}, \overline{h})$, where $\overline{X}$
is the closure of $X$ and $\overline{h}$ is the extension of $h$ to
$\overline{X} \times [0, +\infty)$, the set $Z$ must be closed.\footnote{By injectivity of $Y$, the
  map $h \from X \times [0, +\infty) \to X \subset Y$ extends to a map
  $\overline{h} \from \overline{X} \times [0, +\infty) \to Y$.} It suffices to show that $Z = Y$,
since an appropriate reparametrization then makes $h$ a free deformation contraction of $Y$ to $p$.

Suppose for contradiction that there exists a point $q \notin Z$. We aim to construct a pair
$(Z', h') \succ (Z, h)$ with $q \in Z'$. The idea, following~\cite{isbell64}, is to take the solid
ball $B(p, d(p, q))$, which contains $q$, and extend the restriction of $h$ on
$(Z \cap B(p, d(p, q))) \times [0, d(p, q)]$ to a map
$B(p, d(p, q)) \times [0, d(p, q)] \to B(p, d(p, q))$ so that the trajectory of $q$ is a geodesic
$J$ from $q$ to $p$. Combining this free deformation along the geodesic with $h$, we obtain a free
deformation contraction of $Z \cup J$, contradicting the maximality of $(Z, h)$.

We claim that there exists a metric space $(Y, d)$, a pair $(Z, h)$, and a point
$q \in Y \setminus Z$ such that no $(Z', h') \succ (Z, h)$ satisfies $q \in Z'$. Equivalently, there
exist chains in $\mathcal{X}$ whose maximal elements do not contain $q$. This clearly provides a
counterexample to the second part of the proof.

Let $(Y, d)$ be $(\mathbb{R}^{2}, d_{\infty})$, where
$d_{\infty}((x_{1}, y_{1}), (x_{2}, y_{2})) = \max(|x_{1} - x_{2}|, |y_{1} - y_{2}|)$, and let
$p = (0, 0)$. Let $Z$ be the geodesic from the point $a \defeq (2, 1)$ to $p$, consisting of the two
straight segments $a = (2, 1) \longrightarrow (1, 0) \longrightarrow (0, 0) = p$, as shown
in~\cref{fig:counterexample}. Define a ``monotone'' free deformation contraction $h$ of this geodesic by
\[
  h((x, y), t) = \begin{cases}
    (x, y), &t \geq \max(x, y), \\
    (t, t - 1), &t \leq \max(x, y) \text{ and } t \in [1, 2], \\
    (t, 0), &t \leq \max(x, y) \text{ and } t \in [0, 1].
  \end{cases}
\]
It is straightforward to verify that $(Z, h) \in \mathcal{X}$.

\begin{figure}[tb]
  \centering
  \begin{tikzpicture}[scale=2,line cap=round]
    \draw[step=1,gray,dotted] (-0.5,-0.5) grid (2.5,2.5);

    \draw (-0.7, 0) -- (2.7, 0);
    \draw (0, -0.7) -- (0, 2.7);

    \fill[color=red] (2, 1) circle (1.4pt) node[above right] {$a$};
    \draw[color=red,line width=2pt] (2, 1) -- (1, 0) -- (0, 0);
    \fill[color=red] (0, 0) circle (1.4pt) node[below left] {$p$};
    \fill[color=blue] (2, 2) circle (1.4pt) node[above right] {$q$};
  \end{tikzpicture}
  
  \caption{Counterexample to the proof of Theorem~1.1 in~\cite{isbell64}. The subspace $Z$ is
    colored in red.}\label{fig:counterexample}
\end{figure}

Now take $q \defeq (2, 2)$, and suppose that there exists a pair $(Z', h') \succ (Z, h)$ with
$q \in Z'$. Since $d(p, q) = 2$, we must have $h'(q, 2) = q$. Moreover, for $x \in Z$, we have
$h'(x, s) = h(x, s)$ for every $s$. Hence,
\[
  d_{\infty}(h'(q, 2), h'(a, 1)) = d_{\infty}(q, (1, 0)) = 2 > \max(d_{\infty}(q, a), |2-1|) = 1,
\]
contradicting the non-expansiveness condition required for $(Z', h') \in \mathcal{X}$. Therefore, no
element in any chain of $\mathcal{X}$ containing $(Z, h)$ can contain the point $q$.

\subsection{Free deformation contractibility of compact injective spaces}

While the statement seems to remain a conjecture in the general case, it can be proven for compact
spaces using a different approach. Before describing the proof, we provide some preliminary
definitions.

\begin{definition}
  Let $(X, d)$ be a metric space. A \emph{bicombing} is a set-theoretic map
  \[
    \gamma \from X \times X \times I \to X
  \]
  such that for every $x, y \in X$, the restriction $\gamma(x, y, \cdot)$ is a constant-speed
  geodesic from $x$ to $y$, meaning that
  \begin{enumerate}[nosep]
  \item $\gamma(x, y, 0) = x$,
  \item $\gamma(x, y, 1) = y$,
  \item $d(\gamma(x, y, s), \gamma(x, y, t)) = |s - t|\, d(x, y)$ for all $s, t \in [0,1]$.
  \end{enumerate}

  The bicombing $\gamma$ is called \emph{conical} if, in addition, for all $x, y, x', y' \in X$ and
  $t \in [0,1]$, we have
  \[
    d(\gamma(x, y, t), \gamma(x', y', t)) \leq (1 - t) \, d(x, x') + t \, d(y, y').
  \]

  Moreover, a bicombing $\gamma$ is \emph{consistent} if for all $x, y \in X$ and $s, t \in [0,1]$,
  \begin{enumerate}[nosep]
  \item $\gamma(x, y, t) = \gamma(y, x, 1 - t)$,
  \item $\gamma(x, y, s t) = \gamma(x, \gamma(x, y, t), s)$.
  \end{enumerate}
\end{definition}

Recall that a metric space is called \emph{proper} if every bounded closed set is compact, or
equivalently, if every closed ball is compact. We will use the following result:

\begin{theorem*}[Theorem~1.4 in~\cite{basso_bicomb}]
  Let $(X, d)$ be a proper metric space admitting a conical bicombing. Then there exists a
  consistent bicombing $\gamma$ on $X$ such that, for every $x, y, x', y' \in X$ with
  $d(x, y) = d(x', y')$, the map
  \[
    t \mapsto d(\gamma(x, y, t), \gamma(x', y', t))
  \]
  is convex on $[0, 1]$.
\end{theorem*}

Now, we prove that the set-theoretic map $\gamma \from X \times X \times I \to X$ provided by the
previous theorem is actually continuous:

\begin{lemma}\label{l:bicomb_cont}
  Let $(X, d)$ be a metric space, and let $\gamma$ be a consistent bicombing such that the map
  $t \mapsto d(\gamma(x, y, t), \gamma(x', y', t))$ is convex on $[0, 1]$ for every $x, y, x', y'$
  with $d(x, y) = d(x', y')$. Then $\gamma \from X \times X \times I \to X$ is continuous.
\end{lemma}

\begin{proof}
  Let $\{(x_{n}, y_{n}, t_{n})\}_{n \in \N}$ be a sequence converging to $(x, y, t)$ in
  $X \times X \times I$. It suffices to show that
  $\gamma(x_{n}, y_{n}, t_{n}) \tendsto \gamma(x, y, t)$. Note that
  $d(x_{n}, y_{n}) \tendsto d(x, y)$ as the distance function $d \from X \times X \to \mathbb{R}$ is
  continuous.

  First, consider the case $x = y$. Noting that $\gamma(x, x, t) = x$, we have
  \begin{equation}\label{eq:triangle}
    d(\gamma(x_{n}, y_{n}, t_{n}), x) \leq d(\gamma(x_{n}, y_{n}, t_{n}), x_{n}) + d(x_{n}, x)
    = t_{n} d(x_{n}, y_{n}) + d(x_{n}, x).
  \end{equation}
  Since $d(x_{n}, y_{n}) \tendsto 0$ and $d(x_{n}, x) \tendsto 0$, it follows that
  $d(\gamma(x_{n}, y_{n}, t_{n}), \gamma(x, x, t)) \tendsto 0$.

  Next, consider the case $t = 0$. Since $\gamma(x, y, 0) = x$, we need to show that
  $d(\gamma(x_{n}, y_{n}, t_{n}), x) \tendsto 0$, which follows immediately from~\eqref{eq:triangle}
  because $t_{n} \tendsto 0$ and $d(x_{n}, x) \tendsto 0$. For $t = 1$, since $\gamma(x, y, 1) = y$,
  a similar argument gives
  \[
    d(\gamma(x_{n}, y_{n}, t_{n}), y) \leq (1 - t_{n}) d(x_{n}, y_{n}) + d(y_{n}, y) \tendsto 0.
  \]

  Therefore, we may assume that $x \neq y$ and $0 < t < 1$. We split the sequence
  $\{(x_{n}, y_{n}, t_{n})\}_{n \in \N}$ into two subsequences: one consisting of indices $n$ for
  which $d(x_{n}, y_{n}) \geq d(x, y)$, and the other of indices for which
  $d(x_{n}, y_{n}) < d(x, y)$. Since every term of the original sequence belongs to exactly one of
  these subsequences, it suffices to show that the images of both subsequences converge to
  $\gamma(x, y, t)$, or, if one subsequence is finite, that the images of the other subsequence
  converge to $\gamma(x, y, t)$.

  Consider first the subsequence with $d(x_{n}, y_{n}) \ge d(x, y)$, and assume it is infinite. For
  each sufficiently large $n$, let $s_{n} \defeq \frac{d(x, y)}{d(x_{n}, y_{n})}$. This is
  well-defined since $d(x_{n}, y_{n}) \tendsto d(x, y) \neq 0$, and also $s_{n} \leq 1$. Set
  $\hat{y}_{n} \defeq \gamma(x_{n}, y_{n}, s_{n})$. Then
  \[
    d(x_{n}, \hat{y}_{n}) = d(x_{n}, \gamma(x_{n}, y_{n}, s_{n})) = s_{n}d(x_{n}, y_{n}) = d(x, y),
  \]
  and the functions $f_{n}(\tau) \defeq d(\gamma(x_{n}, \hat{y}_{n}, \tau), \gamma(x, y, \tau))$ are
  convex on $[0, 1]$. Hence,
  \[
    \begin{aligned}
      f_{n}(0) &= d(x_{n}, x) \tendsto 0, \\
      f_{n}(1) &= d(\hat{y}_{n}, y) \leq d(\gamma(x_{n}, y_{n}, s_{n}), y_{n}) + d(y_{n}, y) \\
               &= (1 - s_{n})d(x_{n}, y_{n}) + d(y_{n}, y) = d(x_{n}, y_{n}) - d(x, y) + d(y_{n}, y) \tendsto 0,
    \end{aligned}
  \]
  It follows that $\sup_{\tau \in [0, 1]} f_{n}(\tau) \leq f_{n}(0) + f_{n}(1)$ converges to zero.

  Now, we have
  \begin{multline*}
    d(\gamma(x_{n}, y_{n}, t_{n}), \gamma(x, y, t)) \\
    \begin{aligned}
      &\leq d(\gamma(x_{n}, y_{n}, t_{n}), \gamma(x_{n}, \hat{y}_{n}, t_{n})) + d(\gamma(x_{n}, \hat{y}_{n}, t_{n}), \gamma(x, y, t_{n}))+ d(\gamma(x, y, t_{n}), \gamma(x, y, t)) \\
      &= d(\gamma(x_{n}, y_{n}, t_{n}), \gamma(x_{n}, \gamma(x_{n}, y_{n}, s_{n}), t_{n})) + f_{n}(t_{n}) + d(x, y)|t - t_{n}| \\
      &= d(\gamma(x_{n}, y_{n}, t_{n}), \gamma(x_{n}, y_{n}, s_{n} t_{n})) + f_{n}(t_{n}) + d(x, y)|t - t_{n}| \\
      &= d(x_{n}, y_{n})(t_{n} - s_{n} t_{n}) + f_{n}(t_{n}) + d(x, y)|t - t_{n}|.
    \end{aligned}
  \end{multline*}
  Since $d(x_{n}, y_{n})$ converges to $d(x, y)$ and hence $s_{n} = \frac{d(x, y)}{d(x_{n}, y_{n})}$
  converges to 1, we have $d(x_{n}, y_{n})(t_{n} - s_{n} t_{n}) \tendsto 0$. As
  $\sup_{\tau \in [0, 1]} f_{n}(\tau)$ converges to zero, $f_{n}(t_{n})$ converges to zero as well. Finally,
  $t_{n} \tendsto t$ implies $d(x, y)|t - t_{n}| \tendsto 0$. Therefore,
  $d(\gamma(x_{n}, y_{n}, t_{n}), \gamma(x, y, t))$ converges to zero, as desired.

  Next, consider the subsequence with $d(x_{n}, y_{n}) < d(x, y)$, and assume it is infinite. Set
  $s_{n} \defeq \frac{d(x_{n}, y_{n})}{d(x, y)} < 1$ and $y^{(n)} \defeq \gamma(x, y, s_{n})$.
  As in the previous case, we have
  \[
    d(x, y^{(n)}) = d(x, \gamma(x, y, s_{n})) = s_{n}d(x, y) = d(x_{n}, y_{n}),
  \]
  so the functions $f_{n}(\tau) \defeq d(\gamma(x_{n}, y_{n}, \tau), \gamma(x, y^{(n)}, \tau))$ are
  convex on $I$, and 
  \[
    \begin{aligned}
      f_{n}(0) &= d(x_{n}, x) \tendsto 0, \\
      f_{n}(1) &= d(y_{n}, y^{(n)}) \leq d(y_{n}, y) + d(y, \gamma(x, y, s_{n})) \\
               &= d(y_{n}, y) + (1 - s_{n}) d(x, y) = d(y_{n}, y) + d(x, y) - d(x_{n}, y_{n}) \tendsto 0,
    \end{aligned}
  \]
  Hence $\sup_{\tau \in [0, 1]} f_{n}(\tau)$ converges to zero. Therefore,
  \begin{multline*}
    d(\gamma(x_{n}, y_{n}, t_{n}), \gamma(x, y, t)) \\
    \begin{aligned}
      &\leq d(\gamma(x_{n}, y_{n}, t_{n}), \gamma(x, y^{(n)}, t_{n})) + d(\gamma(x, y^{(n)}, t_{n}),
        \gamma(x, y, t_{n})) + d(\gamma(x, y, t_{n}), \gamma(x, y, t)) \\
      &= f_{n}(t_{n}) + d(x, y)(t_{n} - t_{n}s_{n}) + d(x, y)|t_{n}-t| \tendsto 0.
    \end{aligned}
  \end{multline*}

  This concludes the proof.
\end{proof}

We are now ready to prove Isbell's statement for the proper spaces.

\begin{theorem}\label{theorem:injective_metric}
  Every proper injective metric space is freely deformation contractible to each of its point.  
\end{theorem}

\begin{proof}
  Let $(X, d)$ be a proper injective metric space. By Proposition~3.8 in~\cite{urs_lang}, $X$ admits
  a conical bicombing. Then Theorem~1.4 in~\cite{basso_bicomb} provides a consistent bicombing
  $\gamma \from X \times X \times I \to X$ such that $t \mapsto d(\gamma_{xy}(t), \gamma_{x'y'}(t))$
  is convex whenever $d(x, y) = d(x', y')$. By~\cref{l:bicomb_cont}, $\gamma$ is continuous.

  Fix $p \in X$ and define $h \from X \times [0, +\infty)$ by
  \[
    h(x, t) = \begin{cases}
      p, & x = p, \\
      \gamma\left(p, x, \min\left\{\frac{t}{d(p, x)}, 1\right\}\right), & x \neq p. \\
    \end{cases}
  \]

  Observe that $h(x, 0) = p$ and $h(x, t) = x$ whenever $t \geq d(p, x)$. Moreover, for $x \neq p$,
  we have
  \[
    d(p, h(x, t)) = \begin{cases}
      d(p, x), &t \geq d(p, x), \\
      d\left(p, \gamma\left(p, x, \frac{t}{d(p, x)}\right)\right) = \frac{t}{d(p, x)}d(p, x) = t, &t \leq d(p, x),
    \end{cases}
  \]
  and hence $d(p, h(x, t)) = \min(t, d(p, x))$.

  Next, we verify a form of the freeness condition: for all $s, t \geq 0$, we have
  $h(h(x, t), s) = h(x, \min(s, t))$. For $x = p$ or $\min(s, t) = 0$, both sides equal to $p$. Assuming
  $x \neq p$, we compute
  \begin{align*}
    h(h(x, t), s) & = \gamma\left(p, \gamma\left(p, x, \min\left\{ \frac{t}{d(p, x)}, 1
                    \right\}\right), \min\left\{ \frac{s}{d(p, h(x, t))}, 1 \right\}\right) \\
                  & = \gamma\left(p, x, \min\left\{ \frac{t}{d(p, x)}, 1 \right\} \cdot
                    \min\left\{\frac{s}{\min(t, d(p, x))}, 1 \right\}\right) \\
                  & = \begin{cases}
                    \gamma\left(p, x, \min\left\{\frac{s}{d(p, x)}, 1\right\}\right) =
                    \gamma\left(p, x, \min\left\{\frac{\min(s, t)}{d(p, x)}, 1\right\}\right), & t \geq d(p, x), \\
                    \gamma\left(p, x, \frac{t}{d(p, x)} \cdot \min\left\{\frac{s}{t}, 1\right\}\right)
                    = \gamma\left(p, x, \min\left\{\frac{\min(s, t)}{d(p, x)}, 1\right\}\right), & t \leq d(p, x).
                  \end{cases}
  \end{align*}

  Hence,
  \[
    h(h(x, t), s) = \gamma\left(p, x, \min\left\{\frac{\min(s, t)}{d(p, x)}, 1\right\}\right) = h(x,
    \min(t, s)).
  \]

  Finally, observe that as $t_{n} \tendsto +\infty$ and $x_n \tendsto x$, the sequence
  $h(x_n, t_{n})$ converges to $x$. Therefore, $h$ extends continuously to a map
  $\hat{h} \from X \times [0, +\infty] \to X$, where $[0, +\infty]$ denotes the one-point
  compactification of $[0, +\infty)$, by setting $\hat{h}(x, +\infty) = x$.

  Next, let $\theta \from [0, 1] \to [+\infty, 0]$ be any monotone homeomorphism with
  $\theta(0) = +\infty$ and $\theta(1) = 0$, for example $\theta(t) = \frac{1-t}{t}$ with
  $\theta(0) = +\infty$, and define $H(x, t) \defeq h(x, \theta(t))$.

  By construction, $H(x, 0) = x$ and $H(x, 1) = p$. Moreover, the ``freeness'' of $h$ ensures that
  \[
    H(H(x, t), s) = h(h(x, \theta(t)), \theta(s)) = h(x, \min(\theta(t), \theta(s))) = H(x, \max(t, s)).
  \]
  Hence, $H$ is a free deformation contraction of $X$ to $p$.
\end{proof}

The above proof applies only to proper spaces; therefore, in the general case, Isbell's statement
remains a conjecture.\footnote{It seems to be unknown whether Theorem~1.4 in~\cite{basso_bicomb} can
  be generalized to the general case. This question is related to Descombes-Lang's question of
  whether every geodesic metric space admitting a conical bicombing also admits a consistent convex
  bicombing (\cite{descombes-lang}, see also Section~7 in~\cite{basso2}).} Nevertheless, the result
clearly holds for compact polyhedra, since compact spaces are proper: if a compact polyhedron $P$
admits an injective metric, then $P$ admits a free deformation contraction to a point. Moreover, if
under suitable conditions the resulting contraction is piecewise-linear, then, in view
of~\cref{theorem:main_theorem}, we would obtain an alternative characterization of collapsibility in
metric terms.

\section{Discussion}

In this section, we discuss some open questions and possible directions for further research. We
begin by considering questions related to some aspects of our main result.

In the proof of~\cref{theorem:main_theorem}, the simplicial collapsing of a chosen triangulation of
$P \times I$ produces a polyhedral collapsing in the image under the free deformation retraction
$h$. One may ask whether this can be arranged so that both collapsing are simplicial. In other
words, can we choose triangulations of $P \times I$ and of $P$ so that the simplicial collapsing in
the domain induces a simplicial collapsing in the codomain? This seems to be connected to a natural
question: under what conditions can two piecewise-linear maps with the same domain be simultaneously
triangulated? The only result known to the author in this direction does not appear to apply to our
setting~\cite{bryant}.

Another question is whether the hypotheses on the deformation retraction
in~\cref{theorem:main_theorem} can be weakened. It is clear that, in general, the
piecewise-linearity assumption cannot be omitted, since there exist non-collapsible polyhedra that
are topologically homeomorphic to balls~\cite{berstein1978contractible}. However, in dimensions one
and two the theorem remains valid even without assuming piecewise-linearity~\cite{isbell64}. The
author is not aware of any analogous results in dimension three, which is particularly relevant in
view of the Zeeman conjecture.

In this connection, the following question posed by Sergey Melikhov in a private discussion is also
of interest. A strict deformation retraction $h \from P \times I \to P$ is called \emph{semi-free}
if $h_{s} \circ h_{t} = h_{t}$ whenever $t \geq s$, or equivalently,
$h_{s}\rest{h_{t}(P)} = \id_{h_{t}(P)}$ for $t \geq s$. Does~\cref{theorem:main_theorem} remain true
when free deformation retractibility is replaced by semi-free deformation retractibility?

Finally, it is natural to ask whether there exist invariant metric characterisations of
collapsibility. The existence of non-invariant characterisations, together with the connection
between free deformation retractability and injective metrizability discussed in the previous
section, suggests that such characterisations may indeed exist. For instance, one may ask what
(invariant) conditions on an injective metric guarantee that it yields a piecewise-linear free
contraction. This question is likely closely related to the problem of understanding the
relationship between conical bicombings and consistent conical bicombings; see also the remarks
following Question~1 in~\cite{basso}.

\section*{Acknowledgements}

I express my deep gratitude to Sergey Melikhov for introducing me to this subject and for the
invaluable discussions and guidance provided throughout my research. I am also very grateful to
Francis Lazarus and Martin Deraux for their insightful feedback and criticism on the final version
of this text, and to Giuliano Basso for kindly answering questions regarding bicombings.

\printbibliography

\end{document}